\documentclass{article}
\usepackage[utf8]{inputenc}
\usepackage[dvipsnames]{xcolor}
\usepackage{amsthm}
\usepackage{amsmath}
\usepackage{enumitem}
\usepackage{algorithm}
\usepackage{physics}
\usepackage{amssymb}
\usepackage{algpseudocode}
\usepackage{mathtools}
\usepackage[T1]{fontenc}
\usepackage{amsfonts}
\usepackage{graphicx, epstopdf}
\usepackage{hyperref}
\usepackage{enumitem} 
\usepackage{caption}
\usepackage{mathrsfs}
\usepackage{float}
\usepackage{array}
\usepackage{subcaption}
\usepackage{makecell}

\usepackage[numbers,sort&compress]{natbib}
\bibpunct[, ]{[}{]}{,}{n}{,}{,}

\usepackage{braket}
\usepackage[a4paper, left=3.5cm, right=3.5cm, top=3cm]{geometry}
\newtheorem{theorem}{Theorem}[section]
\newtheorem{definition}[theorem]{Definition}
\newtheorem{lemma}[theorem]{Lemma}
\newtheorem{assumption}[theorem]{Assumption}
\newtheorem{proposition}[theorem]{Proposition}

\newtheorem{remark}[theorem]{Remark}

\newcommand{\R}{\mathbb{R}}
\newcommand{\N}{\mathbb{N}}
\newcommand{\M}{\mathcal{M}_0(\Omega)}
\newcommand{\K}{\mathcal{K}(\Omega)}
\newcommand{\Linf}{L^{\infty}(\Omega)}
\newcommand{\Cinf}{C_c^{\infty}(\Omega)}
\newcommand{\io}{\int_{\Omega}}
\newcommand{\Lmu}{L_{\mu}^2(\Omega)}
\newcommand{\Ltwo}{L^2(\Omega)}

\newcommand{\caps}{\operatorname{cap}_s}

\newcommand{\gammaTo}{\xrightarrow{\gamma}}
\newcommand{\specHs}{\mathbb{H}^s(\Omega)}

\numberwithin{equation}{section}

\begin{document}
	
	\title{Capacitary measures in fractional order Sobolev spaces: Compactness and applications to minimization problems}

	\author{Anna Lentz%
		\thanks{Institut f\"ur Mathematik,
			Universit\"at W\"urzburg,
			97074 W\"urzburg, Germany, {\tt anna.lentz@uni-wuerzburg.de}.
			This research was partially supported by the German Research Foundation DFG under project grant Wa 3626/5-1.}}
	
	\date{\today}
	
	\maketitle
	
	{\bfseries Abstract.}
	Capacitary measures form a class of measures that vanish on sets of capacity zero. These measures are compact with respect to so-called $\gamma$-convergence, which relates a sequence of measures to the sequence of solutions of relaxed Dirichlet problems. This compactness result is already known for the classical $H^1(\Omega)$-capacity.
		This paper extends it to the fractional capacity defined for fractional order Sobolev spaces $H^s(\Omega)$ for $s\in (0,1)$. The compactness result is applied to obtain a finer optimality condition for a class of minimization problems in $H^s(\Omega)$.
	\bigskip

	{\bfseries Keywords. } Fractional Sobolev spaces, capacities, $\gamma$-convergence, optimality condition, $L^p$-functionals
	
	\bigskip
	
	{\bfseries MSC (2020) classification. }
	49K30, 28A33, 31A15 

	\bigskip

	\section*{Introduction}
	The aim is to derive a compactness result for capacitary measures and to use it to obtain a necessary optimality condition for minimization problems in fractional order Sobolev spaces.
	Compactness of capacitary measures is shown with respect to so-called $\gamma$-convergence.\\
	Here, $\gamma$-convergence is a notion of convergence of measures that is defined via the convergence of solutions of a corresponding relaxed Dirichlet problem, see \cite{29capCompactness,115DirInPerf}.  This convergence is related to the well-known $\Gamma$-convergence of functionals \cite{114Wiener, 113varyingObstacles, 160introGamma}.\\
	To be precise, a sequence of capacitary measures $(\mu_k)$ $\gamma$-converges to $\bar\mu$ if for every $f\in H^s(\Omega)^*$ with $s\in(0,1]$ the solutions $w_k \in H^s(\Omega)\cap L_{\mu_k}^2(\Omega)$ of 
	\begin{align*}(w_k,v)_{H^s(\Omega)} + \int_{\Omega}w_k v \dd \mu_k = \braket{f,v}_{H^s(\Omega)} \qquad \forall v\in H^s(\Omega) \cap L_{\mu_k}^2(\Omega) \end{align*} converge weakly in $H^s(\Omega)$ to the solution $\bar w \in H^s(\Omega)\cap L_{\bar\mu}^2(\Omega)$ of 
	\begin{equation*}
		(\bar w,v)_{H^s(\Omega)} + \int_{\Omega}\bar w v \dd \bar\mu = \braket{f,v}_{H^s(\Omega)} \qquad \forall v\in H^s(\Omega) \cap L_{\bar\mu}^2(\Omega).
	\end{equation*}
	
	For the case $s=1$ it was shown that capacitary measures are compact with respect to $\gamma$-convergence, see e.g. \cite{29capCompactness, 115DirInPerf}. In the fractional case, we did not find such a convergence result in the literature. However, there are related results about
	$\gamma$-convergence for sets, see \cite{110optimalPartFrac}. 
	Therefore, an aim of this paper is to extend the compactness result for $\gamma$-convergence of capacitary measures in for $s=1$
	to the fractional case. \\
	Capacitary measures form a class of measures that vanish on sets of capacity zero. For an open Lipschitz domain $\Omega\subset \R^d$ and a compact set $K\subset \Omega$, we define the fractional capacity of $K$ for $s\in(0,1)$ as 
	\begin{align*}  \caps(K) \coloneqq \inf \{\|w\|_{H^s(\Omega)}: w \in \Cinf, w\geq 1_K \}, \end{align*}
	where $H^s(\Omega)$ is a fractional order Sobolev space. \\
	Some results on fractional capacities can be found for example in \cite{30fracRelCap, 126fracCap, 148quantStab}.
	In \cite{30fracRelCap}, a fractional relative capacity on $\bar\Omega$ is investigated which is then used to characterize zero trace fractional Sobolev spaces. \cite{126fracCap} uses fractional capacities to characterize fractional Sobolev embeddings. \\
	There are various applications of capacitary measures both for $s<1$ and $s=1$.
	In \cite{35genDerObstacle}, capacitary measures are used to characterize generalized derivatives of the solution operator of an obstacle problem.
	One can also use capacitary measures to investigate optimization problems depending on the domain like shape optimization problems \cite{121shapeOpNonloc} or optimal partition problems \cite{110optimalPartFrac}. The measures also allow to solve problems with rather complicated domains like perforated domains as investigated in \cite{29capCompactness,115DirInPerf}.\\
	In this paper we want to use fractional capacity theory and the obtained compactness with respect to $\gamma$-convergence to derive a more detailed optimality condition of an optimization problem considered in \cite{sparse}. This problem is of the form
	\begin{equation} \label{eq:minProbSparse}
		\underset{w\in H^s(\Omega)}{\min}\,\, F(w) + \frac{\alpha}{2} \|w\|_{H^s(\Omega)}^2+\beta \int_{\Omega} |w|^p \dd x,  
	\end{equation}
	with $p\in[0,1)$.
	In the optimality condition of that problem, a multiplier $\bar \lambda \in H^s(\Omega)^*$ corresponding to the non-smooth $L^p$-pseudo-norm can be approximated by a sequence $\lambda_k=w_k\mu_k$, where $\mu_k$ is a capacitary measure and $w_k$ the solution to some approximating auxiliary problem of \eqref{eq:minProbSparse}. Thus, the aim is to pass to the limit in this product and to also obtain a decomposition $\bar\lambda = \bar w\bar\mu$, where $\bar w$ is the solution of the considered optimization problem and $\bar\mu$ is the $\gamma$-limit of the sequence $(\mu_k)$.
	
	Such a decomposition can also be obtained by defining $\bar \mu$ directly via the multiplier $\bar \lambda$. However, this only works under some additional assumptions. Something similar was done in \cite{137convOc} to describe subgradients of a solution operator of a variational inequality.\\ 
	The structure of the paper is as follows. First, we collect some definitions and results on fractional order Sobolev spaces in Section \ref{sec:fracSob}. Fractional capacities and capacitary measures are considered in the following Section \ref{sec:cap}. Then compactness of capacitary measures with respect to $\gamma$-convergence is proved in Section \ref{sec:compact}. The results are applied in Section \ref{sec:application} to obtain a finer optimality condition for some minimization problem. Finally, in Section \ref{sec:numerics} we give some numerical examples.

	\subsection*{Notation}
	Throughout this paper, $\Omega\subset \R^d$ denotes a bounded Lipschitz domain. 
	For some given measure $\mu$, the space $\Lmu$ is defined by 
	\[\Lmu \coloneqq \{w\colon \Omega \to \R : \io |w|^2 \dd \mu < \infty\} \qquad \text{with} \qquad \|w\|_{\Lmu} \coloneqq \left( \io |w|^2 \dd \mu \right)^{\frac{1}{2}}. \] 
	The space $W$ will be defined in the next section as a fractional Sobolev space on $\Omega$ equipped with a suitable inner product.
	Moreover, we use the notation $(\cdot,\cdot)_V$ for the inner product of a Hilbert space $V$ and $\braket{\cdot,\cdot}_V$ for its duality product. We denote the negative or positive part of a function $w$ by $w_-$ and $w_+$, respectively.\\
	For a set $A$ we denote by $I_A$ its characteristic function that is one on $A$ and zero elsewhere.
	
	\section{Fractional Sobolev spaces}\label{sec:fracSob}
	To begin, we briefly introduce fractional order Sobolev spaces and state some auxiliary results. There are various definitions of fractional Sobolev spaces in the literature and we consider four of them here.  
	\begin{definition} Let $s\in (0,1)$.
		Then the fractional order Sobolev space $H^s(\Omega)$
		is defined as $$H^{s}(\Omega) \coloneqq  \left\{ w \in L^2(\Omega): \frac{|w(x)-w(y)|}{|x-y|^{\frac{d}{2}+s}} \in L^2(\Omega \times \Omega) \right\}  $$
		with norm \begin{align}\label{eq:def_frac_int}
			\|w\|_{H^{s}(\Omega)} \coloneqq \left(\int_{\Omega}|w|^2\dd x+\frac{c_{d,s}}{2}\int_{\Omega}\int_{\Omega} \frac{|w(x)-w(y)|^2}{|x-y|^{d+2s}}\dd y \dd x \right)^{\frac{1}{2}}\end{align} 
		and inner product
		\[
		(u,w)_{H^s(\Omega)}  = \io uw \dd x + \frac{c_{d,s}}{2} \io \io \frac{(u(x)-u(y)) (w(x)-w(y)))}{|x-y|^{d+2s}} \dd y \dd x,
		\]
		where
		$c_{d,s}\coloneqq \frac{s 2^{2s}\Gamma(s+\frac{d}{2})}{\pi^{\frac{d}{2}}\Gamma(1-s)}.$ 
	\end{definition}
	
	Using this norm, the space $H_0^s(\Omega)$ is defined as $$ H_0^s(\Omega) \coloneqq \overline{C_c^{\infty}(\Omega)}^{H^s(\Omega)}, $$ where $C_c^{\infty}(\Omega)$ is the space of infinitely often continuously differentiable functions with compact support in $\Omega$.
	Furthermore, let 
	\begin{equation}\label{eq:intHs}
		\tilde{H}^s(\Omega) \coloneqq \left\{ w \in H^s(\mathbb{R}^d):  w_{|\mathbb{R}^d\setminus  \Omega} = 0 \right\}=\overline{C_c^{\infty}(\Omega)}^{H^s(\mathbb{R}^d)},
	\end{equation}
	where the second identity was shown in \cite[Theorem 6]{7densityFrac}.
	The space $\tilde{H}^s(\Omega)$ is supplied with the $H^s(\R^d)$-inner product.
	
	Besides these integral fractional Sobolev spaces, we also consider the spectral fractional Sobolev space. This space is defined via eigenvectors and eigenfunctions of the Laplace operator with Dirichlet boundary conditions. 
	\begin{definition}\label{def:specHs} Let $s\geq 0$. 
		Then the fractional order Sobolev space in spectral form is defined as  $$ \mathbb{H}^s(\Omega) \coloneqq \left\{ w= \sum_{n=1}^{\infty} \left( \int_{\Omega}w \phi_n \dd x \right) \phi_n \in L^2(\Omega): ||w||_{\mathbb{H}^s}^2 \coloneqq \sum_{n=1}^{\infty} \lambda_n^s w_n^2<\infty \right\},$$
		where $\phi_n$ are the eigenfunctions of the Laplacian with zero Dirichlet boundary conditions to the eigenvalues $\lambda_n$.
	\end{definition}
	This norm can be equivalently expressed using the integral formulation  
	\begin{align} \label{eq:spectralInt}\|w\|^2_{\mathbb{H}^s(\Omega)}= \frac{1}{2} \int_{\Omega}\int_{\Omega}|(w(x)-w(y))|^2J(x,y) \dd y \dd x + \int_{\Omega}\kappa(x) |w(x)|^2\dd x, 
	\end{align}
	as done for example in \cite{nonhomBC, anote}.
	Here, $J$ and $\kappa$ are measurable non-negative functions, and $J$ is symmetric with $J(x,y)=J(y,x)$  for a.a. $x,y \in \Omega$.
	
	In \cite{105ellipticProb, lions} it was shown that these three spaces coincide for most $s\in(0,1)$ and have equivalent norms. To be precise, it holds 
	\begin{align} 
		\label{eq:equiv_spaces}H_0^s(\Omega) = \left\{ \begin{matrix} H^s(\Omega) = \tilde{H}^s(\Omega) =\specHs && \textrm{  if  }0<s<\frac{1}{2},  \\ H^s(\Omega) && \textrm{  if  } s=\frac{1}{2}, \\ \tilde{H}^s(\Omega) =\specHs && \textrm{  if  } \frac{1}{2}<s<1,  \end{matrix}\right. 
	\end{align} 
	From now on, the space $W$ denotes the fractional Sobolev space $H_0^s(\Omega)$ for some $s\in (0,1)$ equipped with one of the equivalent norms of spaces that coincide with $H_0^s(\Omega)$ by statement \eqref{eq:equiv_spaces} above.\\
	Furthermore, the following properties of these fractional Sobolev spaces were shown in \cite[Section 6]{sparse} and \cite[Sections 6, 7]{hitch}. 
	\begin{proposition} \label{prop:propFrac} The space $W$ satisfies the following properties.
		\begin{enumerate}
			\item[(i)] The embedding $W\hookrightarrow L^2(\Omega)$ is compact.
			\item[(ii)] There is $q>2$ such that $W$ is continuously embedded in $L^q(\Omega)$. More precisely, \begin{align*}
				W \hookrightarrow \left\{ \begin{matrix}
					L^{\frac{2d}{d-2s}}(\Omega) & \text{if } d>2s, \\ L^p(\Omega), p\in[1,\infty) & \text{if } d=2s, \\ C^{0,s-\frac{d}{2}}(\Bar{\Omega}) & \text{if } d<2s.
				\end{matrix} \right.
			\end{align*}
		\end{enumerate}
	\end{proposition}
	
	Next, we state some auxiliary results that we use in the upcoming sections.
	\begin{lemma}\label{lm:LinfDenseW}
		$\Linf$ is dense in $W^*$.
	\end{lemma}
	\begin{proof} This follows from density of $\Linf$ in $H^{-1}(\Omega)$ and $W^* \subset H^{-1}(\Omega)$.
	\end{proof}
	
	\begin{lemma}\label{lm:inequPosPart}
		Let $w\in W$. Then $ \|w_+\|_W^2 \le (w,w_+)_W \le \|w\|_W^2$.
	\end{lemma}
	\begin{proof}
		This is similar to \cite[Lemma 2.6 (b)]{30fracRelCap}. One can directly see that 
		\[\io w_+^2 \dd x \le \io w w_+ \dd x \le \io w^2\dd x,\]
		so the inequalities hold for the $L^2(\Omega)$-part in the $W$-norm. 
		For the first inequality, one can easily see that $(w_+(x)-w_+(y))^2\le (w_+(x)-w_+(y))\cdot (w(x)-w(y))$ a.e. in $\Omega$. The result then follows from the definition of the inner products for $W$, using the integral formulation in the spectral case $\specHs$.     
		The second inequality was shown in \cite[Lemma 2.4]{spatSparse}. 
	\end{proof}
	
	\begin{lemma}\cite[Remark 2.5]{30fracRelCap}\label{lm:prodOfWFuns}
		Let $w,z \in W\cap L^{\infty}(\Omega)$. Then it holds $wz\in W$.
	\end{lemma}
	
	\begin{lemma}\label{lm:zDivW}
		Let $w,z\in W$, $w\in\Linf$ and $z\ge \epsilon$ for some $\epsilon>0$. Then it holds $\frac{w}{z}\in W$.
	\end{lemma}
	\begin{proof}
		Since this result is independent of the choice for $W$, it suffices to consider the integral fractional Sobolev space $H_0^s(\Omega)$. Due to the previous Lemma \ref{lm:prodOfWFuns}, it is also enough to show $\frac{1}{z}\in W$. This follows from
		\begin{align*} \left\|\frac{1}{z}\right\|_{H^s(\Omega)}^2 &= \int_{\Omega}\int_{\Omega} \frac{\left(\frac{1}{z(x)}-\frac{1}{z(y)}\right)^2}{|x-y|^{d+2s}}\dd y\dd x = \int_{\Omega}\int_{\Omega} \frac{\left(\frac{z(y)-z(x)}{z(x)z(y)}\right)^2}{|x-y|^{d+2s}}\dd y\dd x  \\
			&\le \frac{1}{\epsilon^2}  \int_{\Omega}\int_{\Omega} \frac{(z(y)-z(x))^2}{|x-y|^{d+2s}}\dd y\dd x = \frac{1}{\epsilon^2} \|z\|_{H^s(\Omega)}^2 < \infty .\end{align*}
	\end{proof}

	\begin{lemma}\cite[Lemma 2.8]{30fracRelCap}\label{lm:maxMinInW}
		Let $w,z  \in W$ with $w,z \ge 0$. Then $\max(w,z)$ and $\min(w,z)$ are in $W$.
	\end{lemma}
	
	\begin{lemma} \label{lm:phiFracInCinf}
		Let $\varphi\in C_c^{\infty}(\R^d)$. Then it holds \[x \mapsto \int_{\R^d} \frac{(\varphi(x)-\varphi(y))^2}{|x-y|^{d+2s}}\dd y \qquad \in L^{\infty}(\Omega) \]and thus also 
		\[x \mapsto \int_{\Omega} \frac{(\varphi(x)-\varphi(y))^2}{|x-y|^{d+2s}}\dd y \qquad \in \Linf.\]
		For the spectral case, i.e. if $W$ is equipped with the $\specHs$-inner product, it also holds 
		\[x \mapsto \int_{\Omega} (\varphi(x)-\varphi(y))^2 J(x,y) \dd y \qquad \in L^{\infty}(\Omega).\]
	\end{lemma}
	\begin{proof} 
		Let $\varphi \in C_c^{\infty}(\R^d)$, $x\in\Omega$. 
		For each $z\in B_1(0)$, define $g\colon [0,1] \to \R, g(t)=\varphi(x+tz)$. Then 
		\begin{multline*} \varphi(x+z)-\varphi(x) = g(1)-g(0) = \int_0^1g'(t) \dd t = \int_0^1 \nabla \varphi(x+tz)\cdot z \dd t \\
			\le \int_0^1 | \nabla \varphi(x+tz)| |z| \dd t \le \|\nabla \varphi\|_{L^{\infty}(\R^d)} |z|. \end{multline*}
		Next, we note that 
		\[ \int_{B_1(0)} \frac{1}{|z|^{d+2(s-1)}}<\infty \qquad \text{and} \qquad \int_{B_1(0)^c} \frac{1}{|z|^{d+2s}}<\infty. 
		\]
		Using the above estimates together with change of variables, we obtain
		\begin{align*}
			\int_{B_1(x)} \frac{(\varphi(x)-\varphi(y))^2}{|x-y|^{d+2s}} \dd y &= \int_{B_1(x)} \frac{(\varphi(x)-\varphi(x+z))^2}{|z|^{d+2s}}\dd z \\
			&\le \int_{B_1(0)} \frac{\|\nabla\varphi\|_{L^{\infty}(\R^d)}^2}{|z|^{d+2(s-1)}}\dd t\dd z \le c \|\nabla\varphi\|_{L^{\infty}(\R^d)}^2
		\end{align*} 
		for some constant $c\ge 0$. Furthermore, also 
		\[  \int_{\R^d \setminus B_1(0)} \frac{(\varphi(x)-\varphi(y))^2}{|x-y|^{d+2s}} \dd y \le \int_{\R^d \setminus B_1(0)} \frac{4 \|\varphi\|_{L^{\infty}(\R^d)}^2}{|x-y|^{d+2s}} \dd y  \le C  \|\varphi\|_{L^{\infty}(\R^d)}^2 \]
		is bounded.
		
		Since $x$ was arbitrary, the statement for the integral fractional Sobolev spaces follows.\\
		For the spectral case, we use the estimate from \cite[Eq. (33)]{nonhomBC} to obtain
		\[ \int_{\Omega } (\varphi(x)-\varphi(y))^2 J(x,y) \dd y \le c  \int_{\Omega } \frac{(\varphi(x)-\varphi(y))^2}{|x-y|^{d+2s}}\dd y.\]
		Thus, the result follows from the estimations above.
	\end{proof}

	\section{Fractional capacities}\label{sec:cap}
	Next, some definitions and results about fractional capacities and capacitary measures are given.
	Let $\Omega$ be a bounded Lipschitz domain and we choose an inner product for $W$ as described in the previous section. 
	The reason to consider different choices of inner products is that later on we investigate a minimization problem where the minima depend on the choice of the inner product.
	
	\begin{definition}
		Let $K\subset \Omega$ be compact. For the fractional Sobolev space $W$, the fractional capacity of the set $K$ is defined as   \begin{align*}  \caps(K) \coloneqq \inf \{\|w\|_{W}^2: w \in \Cinf, w\geq 1_K \}.  \end{align*}
	\end{definition}
	From this definition one can observe that for different choices of possible inner products for W, equivalence of norms translates to equivalence of capacities.
	This justifies to treat the different choices of inner products simultaneously in the following analysis most of the time. \\
	The fractional capacity can be expressed equivalently as presented in \cite[Theorem 2.1, subsequent comment]{126fracCap}.
	\begin{lemma}\cite[Theorem 2.1]{126fracCap}\label{lm:equivCaps} Let $K\subset\Omega$ be compact.
		It holds
		\begin{align*}
			\caps(K) &= \inf \{\|w\|_{W}^2: w \in W, w=1 \text{ in a neighbourhood of } K \text{ and } 0\le w\le 1 \text{  a.e. on } \Omega \} \\
			&= \inf \{\|w\|_{W}^2: w \in \Cinf, w=1 \text{ in a neighbourhood of } K \text{ and } 0\le w\le 1 \text{  a.e. on }\Omega \}.
		\end{align*}
	\end{lemma}
	
	Another definition stated in \cite[Lemma 3.5]{30fracRelCap} is also equivalent to the formulation above. 
	\begin{lemma}\label{lm:equivCapsFor30}
		Let $2s\le d$ and $K\subset \Omega$ be compact. Then 
		\[\caps(K) = \inf\{ \|w\|_W^2: w\in W \cap C_c(\bar\Omega), w\ge 1 \text{ on } K\}.\]
	\end{lemma}
	\begin{proof}
		As the infimum is taken over a larger set compared to the set in the definition of $\caps(K)$, one directly obtains ``$\ge$''. For the reverse direction ``$\le$'', note that we can replace $w\ge 1 $ on $K$ in the set on the right-hand side by $w\ge 1_K$ on $\Omega$ due to $w_+\in W\cap C_c(\bar\Omega)$ and $\|w_+\|_W \le \|w\|_W$ for all $w\in W$, see Lemma \ref{lm:inequPosPart}. Proceeding as in the proof of \cite[Theorem 2.1]{126fracCap}, we can assume $w\ge 1$ on a neighbourhood $U$ of $K$. As it also holds $\|\min(w,1)\|_W \le \|w\|_W$, we can assume $0\le w\le 1$. Then the statement follows with Lemma \ref{lm:equivCaps}.
	\end{proof}
	
	For general Borel sets, the notion of capacity is extended as follows.
	\begin{definition}\label{def:capOpen}
		For an open set $O\subseteq \Omega$ and a general Borel set $B\subseteq\Omega$, we define
		\[ \caps(O) \coloneqq \sup \{\caps(K): K\subset O\}  \]
		and
		\[\caps(B) \coloneqq \inf \{\caps(O): B\subseteq O, O \text{ open}.\}.  \]
	\end{definition}
	
	The next Lemma provides an equivalent definition in case $2s\le d$.
	
	\begin{lemma}\label{lm:capDefsOpen}
		Let $2s\le d$ and let $O \subseteq \Omega$ open. Then 
		\begin{align*} 
			\caps(O) &= \inf \{ \|w\|_W^2: w\in W, w\ge 1 \text{ a.e. on } O\} \\
			&= \inf \{ \|w\|_W^2: w\in W, w\ge 1_O \text{ a.e. on } \Omega\}.
		\end{align*}
	\end{lemma}
	\begin{proof}
		This follows from results in \cite{30fracRelCap}: There, a space $\tilde W^{s,2}$ with $\tilde{W}^{s,2}\subseteq H_0^s(\Omega)$ is used. For $2s\le d$ and $\Omega$ being a bounded Lipschitz domain, it holds  $\tilde{W}^{s,2}= H_0^s(\Omega)$ by \cite[Theorem 4.8, Example 4.11(a)]{30fracRelCap} and thus the results about $\tilde{W}^{s,2}$ from \cite{30fracRelCap} can be applied to our setting as well.
		As our notion of capacity and the one of \cite{30fracRelCap} coincide on compact sets by  Lemma \ref{lm:equivCapsFor30} and \cite[Lemma 3.5]{30fracRelCap}, this also holds on open sets by 
		Definition \ref{def:capOpen} and \cite[page 17]{30fracRelCap}. 
		This proves the first equality.\\  
		Direction ``$\le$'' of the second equality follows as the second set is a subset of the first one, and the reverse direction ``$\ge$'' is a consequence of $\|w_+\|_W \le \|w\|_W$.
	\end{proof}
	
	\begin{remark}
		It is not possible to generalize the previous lemma to the space \[\mathbb{H}^{\frac12}(\Omega) = H_{00}^{\frac{1}{2}}(\Omega)\coloneqq \left\{w \in H^{\frac{1}{2}}(\Omega): \int_{\Omega} \frac{w^2(x)}{\operatorname{dist}(x,\partial \Omega)}\dd x < \infty \right\},\] 
		as choosing sets $O \subseteq \Omega$ that extend up to the boundary leads to difficulties. Indeed, for $O=\Omega$, the sets on the right-hand side in the previous Lemma \ref{lm:capDefsOpen} are empty.
	\end{remark}
	
	Next, 
	We say that a property holds \textit{quasi everywhere} (q.e.) if there is a set $A$ of capacity zero such that the property holds everywhere on $\Omega\setminus A$.  
	
	A function $w\in W$ is called \textit{quasi continuous} (q.c.) if for every $\epsilon >0$ there is an open set $O$ with $\caps(O)<\epsilon$ such that $w$ is continuous on $\Omega \setminus O$.
	Every $w\in W$ has a q.c.\@ representative: This follows for $2s\le d$ by \cite[Theorem 3.7]{30fracRelCap} and for $2s>d$ by the embedding of $W$ into the continuous functions by Proposition \ref{prop:propFrac}.
	This representative is unique up to sets of capacity zero.
	From now on, functions $w\in W$ are identified with their q.c.\@ representative, such that pointwise values of $w$ are defined q.e.\@ on $\Omega$. \\
	A set $A\subseteq\Omega$ is \textit{quasi open} (q.o.) if for every $\epsilon>0$ there is an open set $O_{\epsilon}\subseteq\Omega$ such that $A\cup O_{\epsilon}$ is open and $\caps(O_{\epsilon})<\epsilon$.
	
	For a sequence $w_k\in W$ that converges to $w$ in $W$, there is a q.e.\@ convergent subsequence, see \cite[Lemma 3.8]{30fracRelCap}.
	Note that these notions depend on the choice of $s$.\\
	If $s>d/2$, then $W$ embeds into the continuous functions. Thus, there are no sets of capacity zero apart from the empty set, and therefore also q.o.\@ and q.c.\@ boils down to open and continuous.\\ 
	Next, we define capacitary measures.
	By a Borel measure we refer to a non-negative measure defined on the $\sigma$-algebra of Borel sets. A Radon measure denotes a regular Borel measure.
	\begin{definition}\label{def:capMeasure}
		The set $\M$ of capacitary measures is defined as the set of Borel measures $\mu$ that satisfy
		\begin{enumerate}[label=(\roman*)]
			\item $\mu(B)=0$ for all Borel sets $B\subseteq\Omega$ with $\caps(B)=0$,
			\item $\mu(B)=\inf\{\mu(A): A \text{ quasi open}, B\subseteq A\}$ for every Borel set $B\subseteq\Omega$.
		\end{enumerate}
	\end{definition}
	A measure $\mu$ is said to be in $W^*$ if there exists $f\in W^*$ such that 
	\begin{equation}\label{eq:forMeasureInDual}
		\braket{f,\varphi}_W = \io \varphi \dd\mu \qquad \forall \varphi \in C_c^{\infty}(\Omega).\end{equation}
	\begin{theorem}\label{tm:radonInCap}
		Let $\mu$ be a Radon measure in $W^*$. Then there is a constant $C>0$ such that 
		\[\mu(B) \le C\, \sqrt{\caps(B)} \qquad \text{for all Borel sets }B \subseteq \Omega, \] 
		and it holds $\mu \in \M$.
	\end{theorem}
	
	\begin{proof}
		First, we show $\mu(B)\le \|f\|_{W^*} \sqrt{\caps(B)}$, where $f$ is given by \eqref{eq:forMeasureInDual}.\\
		Let $K\subset \Omega$ be compact. By definition of $\caps(K)$, there exists a sequence $(\varphi_n)\subset \Cinf$ with $\varphi_n\ge 1_K$ and $\|\varphi_n\|_W \to \sqrt{\caps(K)}$. Then
		\begin{align*}\mu(K)=\int_K\dd\mu \le \int_K \varphi_n \dd\mu \le \io\varphi_n  \dd\mu = \braket{f,\varphi_n}_W \le \|f\|_{W^*} \|\varphi_n\|_W \to \|f\|_{W^*} \sqrt{\caps(K)}.\end{align*}
		For an open set $O\subseteq \Omega$, we have 
		\[
		\mu(O) = \underset{K\subset O}{\sup} \mu(K) \le \|f\|_{W^*} \underset{K\subset O}{\sup} \sqrt{\caps(K)} = \|f\|_{W^*} \sqrt{\caps(O)}
		\]
		by inner regularity of Radon measures. For a general Borel set $B\subseteq \Omega$ we obtain by monotonicity of the measure $\mu$
		\[
		\mu(B)\le \underset{O\supset B}{\inf} \mu(O) \le  \underset{O\supset B}{\inf} \|f\|_{W^*} \sqrt{\caps(O)} = \|f\|_{W^*} \sqrt{\caps(B)} 
		\] 
		This shows the first statement of the theorem with $C=\|f\|_{W^*}$ and also implies $(i)$ in Definition \ref{def:capMeasure}.\\
		To show $(ii)$, we note that by regularity of $\mu$ it holds 
		\[\mu(B)=\inf\{\mu(O) : O\supseteq B, O \text{ open }\}, \] 
		which implies ``$\ge$''. The reverse inequality follows directly from monotonicity of $\mu$.
	\end{proof}

	\begin{lemma}\label{lm:approxSetPw}
		For all quasi-open sets $A\subseteq\Omega$ there exists an increasing sequence $(v_k)\in W$ with $v_k\ge 0$ such that $v_k\to 1_A$ pointwise q.e. in $\Omega$. 
	\end{lemma}
	\begin{proof}
		We follow the lines of the proof of \cite[Lemma 2.1]{29capCompactness}. For $d\ge 2s$, let $A\subseteq \Omega$ be q.o.\@ and $(O_n)$ a sequence of open subsets of $\Omega$ with $\caps(O_n)<1/n$ and $A_n\coloneqq A\cup O_n$ open. Then for all $n\in\N$ there is an increasing sequence $(\varphi_k^n)\subset \Cinf$ with $\varphi_k^n \to 1_{A_n}$ pointwise q.e.\@ in $\Omega$ for $k\to\infty$.
		Due to $\caps(O_n,\Omega)<\frac{1}{n}$, by Lemma \ref{lm:capDefsOpen} there exists $w_n\in W$ with $w_n\ge 1$ q.e.\@ on $O_n$, $w_n\ge 0$ and $\|w_n\|_W^2\le \frac{1}{n}$. Here, $w_n\ge 1$ is satisfied q.e.\@ on $O_n$ as we work with the q.c.\@ representative of $w_n$. 
		Therefore, there exists a subsequence still denoted by $(w_n)$ such that $w_n\to 0$ 
		q.e.\@ in $\Omega$. By $\varphi_k^n\le 1_{A_n}$ and $\varphi_k^n\le 1_A$ on $\Omega\setminus O_n$, one obtains $(\varphi_k^n-w_n)_+ \le 1_A$ q.e.\@ in $\Omega$. Let 
		\[ v_k \coloneqq \underset{1\le n\le k}{\max}(\varphi_k^n-w_n)_+, \qquad \psi = \underset{k}{\sup }\, v_k.\]
		Then $v_k$ is an increasing sequence with $v_k\in W$, $v_k\ge 0$ for all $k\in \N$ and it holds $0\le \psi \le 1_A$ q.e.\@ in $\Omega$.
		Furthermore, for all $k\ge n$ it holds $v_k\ge \varphi_k^n-w_n$. With $A\subseteq A_n$, this yields $\psi\ge 1-w_n$ q.e.\@ in $A$. Passing to the limit $n\to \infty$ then implies $\psi\ge 1$ q.e.\@ in $A$, so $\psi=1_A$. \\
		If $d<2s$, then $A$ is already open and we can directly use an increasing sequence $(\varphi_n)\subset \Cinf$ with $\varphi_n \to 1_{A}$ pointwise q.e.\@ in $\Omega$.
	\end{proof}
	
	For $\mu \in \M$ and $f\in W^*$, the relaxed Dirichlet problem is defined as finding $w\in W\cap L_{\mu}^2(\Omega)$ such that 
	\begin{align}\label{eq:relDirProb}(w,v)_W + \int_{\Omega} wv \dd \mu = \braket{f,v}_W \qquad \forall v \in W\cap L_{\mu}^2(\Omega).\end{align}
	
	\begin{theorem}
		Let $f\in W^*$. Then there exists a unique solution $w$ of \eqref{eq:relDirProb}.
	\end{theorem}
	\begin{proof}
		This follows from the Lax-Milgram lemma as in \cite[Theorem 2.2]{29capCompactness}.
	\end{proof}
	
	Next, we define the notion of $\gamma$-convergence for capacitary measures.
	
	\begin{definition}
		A sequence $(\mu_k) \subset \M$ is called $\gamma$-convergent to some capacitary measure $\bar\mu \in \M$, if for every $f\in W^*$ the solutions $w_k \in W\cap L_{\mu_k}^2(\Omega)$ of \begin{align}\label{eq:relDirProbWk} (w_k,v)_W + \int_{\Omega}w_k v \dd \mu_k = \braket{f,v}_W \qquad \forall v\in W\cap L_{\mu_k}^2(\Omega) \end{align} converge weakly in $W$ to the solution $\bar w \in W \cap L_{\bar\mu}^2(\Omega)$ of \begin{equation}\label{eq:relDirProbW}
			(\bar w,v)_W + \int_{\Omega}\bar w v \dd \bar\mu = \braket{f,v}_W \qquad \forall v\in W\cap L_{\bar\mu}^2(\Omega).
		\end{equation}
	\end{definition}
	We denote $\gamma$-convergence of $(\mu_k)$ to $\bar\mu$ by $\mu_k \gammaTo \bar \mu$.
	The name $\gamma$-convergence arises from its relation to $\Gamma$-convergence for functionals. Given a measure $\mu$, let $F_{\mu} \colon \Ltwo \to [0,\infty]$ be defined as
	\[F_{\mu}(w) \coloneqq \begin{cases}
		\frac12 \|w\|_W^2 + \frac12 \|w\|_{\Lmu}^2 & \text{if } w\in W \cap \Lmu, \\
		+\infty & \text{otherwise}.
	\end{cases}
	\]
	One can show that $F_{\mu_k}$ $\Gamma$-converges to $F_{\bar\mu}$ if and only if $\mu_k \gammaTo \bar \mu$ analogously to \cite[Proposition 4.10]{114Wiener} by replacing $H_0^1(\Omega)$ with $W$.\\
	Note that solutions $w_k$ of \eqref{eq:relDirProbWk} are bounded in $W$ by
	\begin{equation}\label{eq:wkBounded}
		\|w_k\|_W \le \|f\|_{W^*}
	\end{equation}
	due to
	\[\|w_k\|_W^2 \le \|w_k\|_W^2+\|w_k\|_{L_{\mu_k}^2(\Omega)}^2 = \braket{f,w_k}_W\le \|f\|_{W^*} \|w_k\|_W. \]

	\section{Compactness of capacitary measures}\label{sec:compact}
	Now, we extend the compactness result for capacitary measures with respect to $\gamma$-convergence for the case $s=1$ from \cite{29capCompactness} to the fractional case $s<1$. The proofs are quite similar to those presented in \cite{29capCompactness}, but for the sake of completeness we state them here adapted to the fractional setting.
	The underlying idea of the proof is the following: we consider a  weakly compact set $\K \subset W$ and show that functions in $\K$ can be associated with a capacitary measure $\mu$. The main argument to obtain compactness of capacitary measures is then the equivalence of $\gamma$-convergence of a sequence of measures and weak convergence of the sequence of the associated functions in $\K$. \\
	We start with some results that help to characterize the set $\K$ later on.

	\begin{lemma}\label{lm:wNonneg}
		Let $\mu\in\M, f\in W^*$ and let $w\in W\cap L_{\mu}^2(\Omega)$ be the solution of problem \eqref{eq:relDirProb}. Then $f\ge 0$ in $\Omega$ implies $w\ge 0$ q.e.\@ in $\Omega$. 
	\end{lemma}
	\begin{proof} 
		This is an adaption of the proof of \cite[Proposition 2.4]{29capCompactness}.
		Let $v\coloneqq -w_-$, so $v\ge 0$ and $v\in W\cap L_{\mu}^2(\Omega)$. With $wv\le 0$ q.e.\@ and $\braket{f,v}_W\ge 0$, testing \eqref{eq:relDirProb} with $v$ yields $(w,v)_W\ge 0$.
		At the same time one can easily see that
		\[
		0\ge (w(x)-w(y))\cdot ( v(x)-v(y)).
		\]
		By definition of the inner product for $H^s(\Omega)$, this yields $(w,v)_W\le 0$, so $(w,v)_W=0$ using the observation above. Note that these considerations also cover the spectral fractional Sobolev space $\specHs$ due to its integral formulation.
		In particular, this yields $\|v\|_W=0$ by Lemma \ref{lm:inequPosPart}, which implies $v=0$ q.e.\@ on $\Omega$ as we work with the unique q.c. representative of functions in $W$. 
	\end{proof}
	
	\begin{lemma} \label{lm:monOfSolsRelDir}
		Let $f_1,f_2 \in W^*(\Omega)$, $\mu_1, \mu_2 \in \M$ and let $w_1,w_2$ be the solutions of the corresponding relaxed Dirichlet problems \eqref{eq:relDirProb}. Then $0\le f_1\le f_2$ and $\mu_2 \le \mu_1$ implies $0\le w_1 \le w_2$ q.e.\@ in $\Omega$.
	\end{lemma}
	\begin{proof}
		We follow the proof of \cite[Proposition 2.5]{29capCompactness}. Lemma \ref{lm:wNonneg} yields $w_1, w_2\ge 0$ q.e.\@ in $\Omega$. Define $v\coloneqq (w_1-w_2)_+$. As $0\le v\le w_1$ and $\mu_2\le\mu_1$, it holds $v\in L_{\mu_1}^2(\Omega)\cap L_{\mu_2}^2(\Omega)$. Testing the equations for $w_1$ and $w_2$ by $v$ and subtracting them gives together with $\int_{\Omega}w_2 v \dd \mu_2 \le \int_{\Omega} w_2 v \dd \mu_1$
		\[(w_1-w_2,v)_W+\int_{\Omega}(w_1-w_2) v \dd \mu_1 \le \braket{f_1-f_2,v}_W\le 0.\]
		As $(w_1-w_2)v\ge 0$, one obtains with Lemma \ref{lm:inequPosPart} that $\|v\|_W^2\le 0$, so $v=0$ q.e.\@ on $\Omega$.
	\end{proof}
	
	\begin{lemma}\label{lm:estMesaureF}
		Let $\eta\in W^*$ be a Radon measure on $\Omega$ and let $w\in W\cap L_{\mu}^2(\Omega)$ solve problem \eqref{eq:relDirProb} for $f=\eta$. Then it holds 
		\[ (w,v)_W \le \int_{\Omega} v \dd \eta \qquad \forall v\in W \text{ s.t. } v\ge 0 \text{ q.e.\@ in } \Omega.\]
	\end{lemma}
	\begin{proof}
		We generalize \cite[Proposition 2.6]{29capCompactness} from $s=1$ to $s\in(0,1)$. Let $v\in W$ with $v\ge 0$ and define $v_n\coloneqq \min(\frac{1}{n}v,w)$. Lemma \ref{lm:wNonneg} yields $w\ge 0$, and therefore $v_n\ge 0$ and it also holds $v_n\in W\cap L_{\mu}^2(\Omega)$ by Lemma \ref{lm:maxMinInW}. Note that we do not require $v\in \Lmu$.\\
		Now, we test equation \eqref{eq:relDirProb} with $v_n$. With $\int_{\Omega} w v_n \dd \mu \ge 0$, this yields 
		\begin{align}\label{eq:auxIneq}
			(w,v_n)_W\le \int_{\Omega} v_n \dd \eta \le \frac{1}{n} \int_{\Omega} v \dd \eta.	
		\end{align}
		Next, we want to pass to the limit for $n\to\infty$ in the previous inequality.
		To do that, we first note that \begin{equation}\label{eq:l2PartEst}
			\int_{\Omega} w v_n \dd x =  \int_{\Omega\cap \{w>\frac{v}{n}\}} w v_n \dd x +  \int_{\Omega\cap \{w \le \frac{v}{n}\}} w v_n \dd x \ge \frac{1}{n} \int_{\Omega\cap \{w>\frac{v}{n}\}} w v \dd x.
		\end{equation} Then we split the double integral 
		\[ \io\io \frac{(w(x)-w(y)) (\min(\frac{1}{n}v,w)(x)-\min(\frac{1}{n}v,w)(y))}{|x-y|^{d+2s}}\dd y \dd x \qquad (*)\]
		from the $W$-inner product in integral form in four parts:
		\begin{enumerate}[label=(\roman*)]
			\item $x,y\in \{\frac{v}{n} \ge w\}$: Then the nominator is $(w(x)-w(y))^2\ge 0$.
			\item $x,y\in \{\frac{v}{n}< w\}$: The nominator reads $(w(x)-w(y))(\frac{1}{n}v(x)-\frac{1}{n}v(y))$.
			\item $x\in \{\frac{v}{n} < w\}, y\in \{\frac{v}{n} \ge w\}$: 
			\begin{itemize}[label=-, leftmargin=0pt]
				\item 
				For $w(x)\ge w(y)$, it holds \\$(w(x)-w(y))(\frac{1}{n}v(x)-w(y)) \ge (w(x)-w(y))(\frac{1}{n}v(x)-\frac{1}{n}v(y))$.
				\item 
				For $w(x)\le w(y)$, it holds $(w(x)-w(y))(\frac{1}{n}v(x)-w(y)) \ge (w(x)-w(y))^2\ge 0$. 
			\end{itemize}
			\item $x\in \{\frac{v}{n} \ge w\}, y\in \{\frac{v}{n} < w\}$: 
			\begin{itemize}[leftmargin=0pt, label=-]
				\item For $w(x)\ge w(y)$, it holds $(w(x)-w(y))(w(x)-\frac{1}{n}v(y))\ge (w(x)-w(y))^2\ge 0$.
				\item For $w(x)\le w(y)$, it holds\\ $(w(x)-w(y))(w(x)-\frac{1}{n}v(y))\ge (w(x)-w(y))(\frac{1}{n}v(x)-\frac{1}{n}v(y))$.
			\end{itemize}
		\end{enumerate}
		Then one obtains 
		\[(*) \ge \iint_{M_n} \frac{(w(x)-w(y))(\frac{1}{n}v(x)-\frac{1}{n}v(y))}{|x-y|^{d+2d}} = \frac{1}{n} \iint_{M_n} \frac{(w(x)-w(y))(v(x)-v(y))}{|x-y|^{d+2d}} \]
		with 
		\begin{align*} M_n = \left\{ (x, y) \in \left\{\frac{v}{n}< w\right\}\right\} \cup \left\{ (x, y):  \frac{v(x)}{n}< w(x), \, \frac{v(y)}{n} \ge w(y), \, w(x)\ge w(y) \right\} \\ \cup \left\{(x, y): \frac{v(x)}{n} \ge w(x), \, \frac{v(y)}{n} <w(y), \, w(x) \le w(y) \right\} \subseteq \Omega\times\Omega. \end{align*}
		Using \eqref{eq:auxIneq} and \eqref{eq:l2PartEst},  this shows
		\[
		\io v \dd\eta \ge \int_{\Omega\cap \{w>\frac{v}{n}\}} w v \dd x + \iint_{M_n} \frac{(w(x)-w(y))(v(x)-v(y))}{|x-y|^{d+2d}}.
		\]
		Define $M\subseteq \Omega\times\Omega$ as
		\[M\coloneqq \{(x,y) \in \{w>0\}\}\cup \{(x,y): w(x)>0, w(y)=0\}\cup \{(x,y): w(x)=0,w(y)>0\}.
		\]
		By dominated convergence, it is possible to pass to the limit in the previous inequality to obtain 
		\begin{align*}\int_{\Omega} v \dd \eta \ge \int_{\Omega\cap \{w>\frac{v}{n}\}} w v \dd x + \iint_{M_n} \frac{(w(x)-w(y))(v(x)-v(y))}{|x-y|^{d+2d}} \\ \to \int_{\Omega\cap \{w>0\}} w v \dd x + \iint_{M} \frac{(w(x)-w(y))(v(x)-v(y))}{|x-y|^{d+2d}} = (w,v)_W. \end{align*}
		
	\end{proof}

	\begin{lemma}\label{lm:eqNDiff}
		Let $\bar\mu\in\M$, $\bar w\in W\cap L_{\bar \mu}^2(\Omega)$ and for $k\in\N$ let $w_k\in  W\cap L_{\bar \mu}^2(\Omega)$ be the solution of 
		\begin{equation}\label{eq:eqNDiff}
			(w_k,v)_W+\int_{\Omega}w_k v \dd \bar \mu + k \io (w_k-\bar w)v  \dd x = 0\qquad \forall v\in  W\cap L_{\mu}^2(\Omega).
		\end{equation}
		Then $w_k \to \bar w$ in $W$ and in $L_{\mu}^2(\Omega)$. 
	\end{lemma}
	\begin{proof}
		Here, we follow \cite[Proposition 3.1]{29capCompactness}.
		Testing \eqref{eq:eqNDiff} with $v= w_k-\bar w$ yields 
		\[ (w_k, w_k-\bar w)_W+\int_{\Omega}w_k  (w_k-\bar w) \dd \bar \mu + k \io (w_k-\bar w)^2  \dd x = 0. \]
		After adding some terms on both sides we obtain
		\begin{equation}\label{eq:auxWkToW} \|w_k-\bar w\|_W^2 + \int_{\Omega}  (w_k-\bar w)^2 \dd \mu + k \io (w_k-\bar w)^2  \dd x = -(\bar w,w_k-\bar w)_W - \int_{\Omega}\bar w(w_k-\bar w)\dd \bar \mu,\end{equation} which by Cauchy-Schwarz implies
		\[ \|w_k-\bar w\|_W^2 + \|  w_k-\bar w\|_{L_{\mu}^2(\Omega)}^2 + k \|w_k-\bar w\|_{L^2(\Omega)}^2 \le \|\bar w\|_W \|w_k-\bar w\|_W + \|\bar w\|_{L_{\bar \mu}^2(\Omega)} \|w_k-\bar w\|_{L_{\bar \mu}^2(\Omega)} .\]
		This yields 
		\[ \frac1{2} \|w_k-\bar w\|_W^2 + \frac{1}{2} \|  w_k-\bar w\|_{L_{\bar \mu}^2(\Omega)}^2 + k \|w_k-\bar w\|_{L^2(\Omega)}^2 \le  \frac1{2} \|\bar w\|_W^2 + \frac{1}{2} \|  \bar w\|_{L_{\bar \mu}^2(\Omega)}^2.\]
		Thus, we obtain $w_k \to \bar w$ in $\Ltwo$ from the last term on the left-hand side and also $w_k\rightharpoonup \bar w$ in $W$ and $L_{\bar \mu}^2(\Omega)$ by uniqueness of weak limits and the first two terms. Equation \eqref{eq:auxWkToW} then implies strong convergence in $W$ and $L_{\bar \mu}^2(\Omega)$.
	\end{proof}

	\begin{lemma}\label{lm:muInf}
		Let $\mu \in\M$ and $z\in W\cap L_{\mu}^2(\Omega)$ be the solution of 
		\[ (z,v)_W + \int_{\Omega} zv \dd \mu = \int_{\Omega} v \dd x \qquad \forall v \in W\cap L_{\mu}^2(\Omega).\]
		Then $\mu(B)=\infty$ for all Borel sets $B\subseteq \Omega$ with $\caps(B\cap\{z=0\})>0$.
	\end{lemma}
	\begin{proof}
		This proof works as the one of \cite[Lemma 3.2]{29capCompactness}.
		Let $w\in W\cap L_{\mu}^2(\Omega)$ with $0\le w\le 1$ q.e.\@ in $\Omega$ and let $w_n\in  W\cap L_{\mu}^2(\Omega)$ denote the solution of
		\[(w_n,v)_W + \int_{\Omega} w_n v \dd\mu + n\int_{\Omega} w_n v \dd x=n\io wv\dd x \] for all $n\in\N$.
		The comparison principle from Lemma \ref{lm:monOfSolsRelDir}  applied with $f_1=w, f_2=1, \mu_1=\mu+n\dd x$ and $\mu_2=\mu$ yields $0\le\frac{1}{n}w_n\le z$ q.e.\@ on $\Omega$. Thus, $w_n=0$ q.e.\@ in $\{z=0\}$ and therefore also $w=0$ q.e.\@ in $\{z=0\}$ by Lemma \ref{lm:eqNDiff}. \\
		Let $A \subseteq\Omega$ be q.o.\@ with $\mu(A)<\infty$. By Lemma \ref{lm:approxSetPw} there is an increasing sequence $(z_n)\subset W$ converging pointwise to $1_A$ q.e.\@ with $0\le z_n\le 1_A$ q.e.\@ in $\Omega$. Since $\mu(A)<\infty$, $z_n\in\Lmu$ for all $n\in\N$ and therefore $z_n=0$ q.e.\@ on $\{z=0\}$ as shown in the first paragraph. 
		Pointwise convergence $z_n\to 1$ q.e.\@ on $A$ thus shows $\caps(A \cap\{z=0\})=0$. \\
		Let $B$ be a Borel set with $\caps(B\cap \{z=0\})>0$. Then also $\caps(A\cap \{z=0\})>0$ for every q.o.\@ set $A\supseteq B$ and thus
		\[ \mu(B) = \inf\{\mu(A): A \text{ quasi open}, B\subseteq A\}=\infty
		\]
		by definition of $\M$ and the previous step.
	\end{proof}
	
	\begin{lemma}\label{lm:lambdaIsMu}
		Let $\lambda,\mu\in\M$ and let $w\in W\cap L_{\lambda}^2(\Omega)\cap \Lmu$ such that
		\begin{align} \label{eq:lambdaIsMu1}
			(w,v)_W+\io wv \dd \lambda=\io v \dd x\qquad \forall v\in W\cap L_{\lambda}^2(\Omega) \\ \label{eq:lambdaIsMu2}
			(w,v)_W+\io wv \dd \mu=\io v \dd x\qquad \forall v\in W\cap \Lmu.
		\end{align} Then $\lambda=\mu$.
	\end{lemma}
	\begin{proof}
		We repeat the proof of \cite[Lemma 3.3]{29capCompactness} in the fractional setting. Let the measures $\lambda_0, \mu_0$ be defined for every Borel set $B$ by 
		\[ \lambda_0(B) = \int_B w \dd \lambda, \qquad \mu_0(B) = \int_B w \dd\mu,  \]
		and define 
		\[ \lambda_{\epsilon}(B) = \int_{B\cap \{w>\epsilon\}} w \dd \lambda, \qquad \mu_{\epsilon}(B) = \int_{B\cap \{w>\epsilon\}} w \dd\mu.  \]
		We want to show $\lambda_0=\mu_0$, which follows from $\lambda_{\epsilon}=\mu_{\epsilon}$ for all $\epsilon>0$. Let $\epsilon>0$. Since $w\in L_{\lambda}^2(\Omega)\cap \Lmu$, $\lambda_{\epsilon}$ and $\mu_{\epsilon}$ are bounded measures.
		Hence, it is enough to show $\lambda_{\epsilon}(O)=\mu_{\epsilon}(O)$ for all open sets $O\subseteq \Omega$. 
		For such an $O$ we define the q.o.\@ set $A_{\epsilon}\coloneqq O\cap\{w>\epsilon\}$. By Lemma \ref{lm:approxSetPw} one can approximate $1_{A_{\epsilon}}$ by an increasing and non-negative sequence $(z_n)\in W$ that converges pointwise q.e.\@ to $1_{A_{\epsilon}}$ in $\Omega$. By $w\in L_{\lambda}^2(\Omega)\cap \Lmu$ and $w>\epsilon$ q.e.\@ in $A_{\epsilon}$, it holds $\lambda(A_{\epsilon})<\infty$ and $\mu(A_{\epsilon})<\infty$. This implies $z_n\in L_{\lambda}^2(\Omega)\cap \Lmu$. Testing both \eqref{eq:lambdaIsMu1} and \eqref{eq:lambdaIsMu2} with $z_n$ yields $\io wz_n \dd \lambda = \io w z_n \dd\mu$, which in the limit $n\to\infty$ using dominated convergence gives
		\[ \lambda_{\epsilon}(O)=\int_{A_{\epsilon}}w\dd\lambda=\int_{A_{\epsilon}}w\dd\mu=\mu_{\epsilon}(O).\]
		Hence, $\lambda_{\epsilon}=\mu_{\epsilon}$ for all $\epsilon>0$ and therefore $\lambda_0=\mu_0$. To deduce $\lambda=\mu$, we first consider a Borel set $B$ with $B\subseteq\{w>0\}$. Then 
		\[\lambda(B)=\int_B\frac{1}{w}\dd\lambda_0=\int_B\frac{1}{w}\dd\mu_0 = \mu(B).\]
		If $B\subseteq \{w=0\}$ is a Borel set with $\caps(B)>0$, then $\lambda(B)=\mu(B)=\infty$ by Lemma \ref{lm:muInf}. If $\caps(B)=0$, then $\lambda(B)=\mu(B)=0$ as $\lambda,\mu\in\M$. For a general Borel set, 
		\[\lambda(B)=\lambda(B\cap\{w>0\})+\lambda(B\cap\{w=0\})=\mu(B\cap\{w>0\}+\mu(B\cap\{w=0\}=\mu(B).\]
	\end{proof}
	
	\begin{lemma} \label{lm:etaIsRadon}
		Let $0 \le \eta \in W^*$. Then $\eta$ is a Radon measure.
	\end{lemma}
	\begin{proof}
		This was shown in \cite[p. 564]{34perturbationAna}. There, $0\le \eta\in H^{-1}(\Omega)$ is restricted to a linear form on $H^{-1}(\Omega)\cap C_c(\Omega)$ that can be extended to a linear form on $C_{c}(\Omega)$. \cite[Theorem 6.54]{34perturbationAna} shows that this form is a Radon measure. As $W^* \subset H^{-1}$, the claim follows. 
	\end{proof}
	
	Next, we consider the set $\mathcal{K}(\Omega)$ defined as  
	\begin{equation*} \mathcal{K}(\Omega) \coloneqq \left\{ w \in W: w \ge 0  \text{ in } \Omega,\, (w,v)_W \le \io v \dd x \,\, \forall\,\, 0\le v\in W \right\}. \end{equation*} 
	This set is also considered in \cite{121shapeOpNonloc} and is the analogon to the set considered in \cite{29capCompactness} for $s=1$. 
	In \cite[Proposition 3.3]{121shapeOpNonloc} it was shown that this set is convex, closed and bounded in $W$.
	\begin{lemma}\label{lm:KboundedInLinf}
		The set $\K$ is bounded in $\Linf$.
	\end{lemma}
	\begin{proof}
		As in \cite[page 10]{29capCompactness}, let $w_0$ be the solution of 
		\[(w,v)_W =\io v\dd x \qquad \forall \,\, v\in W.\] Then $w_0\in L^{\infty}(\Omega)$ by \cite[Lemma 3.4]{5sNObstacle} with $k_0=0$, $\sigma=(\frac1s+\frac2q-1)^{-1}<0$ and $1\in L^s(\Omega)$, using the embedding $W\hookrightarrow L^q(\Omega)$ from Proposition \ref{prop:propFrac} and the estimate $\|v_k\|_{L^q(\Omega)}^2\le c\|v_k\|_W^2 \le (v_k,v)_W$ with $v_k$ defined as in \cite{5sNObstacle}. 
		Due to Lemma \ref{lm:monOfSolsRelDir} this implies boundedness of $\K$ in $\Linf$.
	\end{proof}
	
	The set $\K$ can be characterized in terms of capacitary measures.
	
	\begin{proposition}\label{prop:charK}
		Let $z\in W$. Then $z\in \K$ if and only if there exists $\mu\in\M$ such that $z\in W\cap\Lmu$ and 
		\begin{equation}\label{eq:zForCharK}(z,v)_W + \io zv\dd\mu = \io v \dd x \qquad \forall v\in W\cap\Lmu.\end{equation}
		This measure $\mu\in\M$ is uniquely determined by $z$ via
		\begin{equation}\label{eq:defMuCharK}
			\mu(B)= \begin{cases}
				\int_B \frac{\dd\eta}{z} & \text{ if } \caps(B\cap\{z=0\})=0, \\ \infty & \text{ if }\caps(B\cap\{z=0\})>0,
		\end{cases}\end{equation}
		where $\eta\in W^*$ is the measure given by $\eta \coloneq 1-(z,\cdot)_W \coloneq (1,\cdot)_{\Ltwo} - (z,\cdot)_W$. Furthermore, for every Borel set $B\subseteq \Omega$ it holds 
		\[\eta(B\cap \{z>0\})=\int_B z\dd\mu.\]
	\end{proposition}
	\begin{proof}
		The proof is done as in \cite[Proposition 3.4]{29capCompactness}. Let $\mu\in\M$ and $z$ the solution of \eqref{eq:zForCharK}. As $z\ge 0$ by Lemma \ref{lm:wNonneg} and $(z, \cdot)_W \le 1$ by Lemma \ref{lm:estMesaureF}, one obtains $z\in\K$.\\
		For the reverse implication, let $z\in\K$ and $\mu$ as defined in \eqref{eq:defMuCharK}. We want to show $\mu\in\M$. As $0\le\eta \in W^*$,  Lemma \ref{lm:etaIsRadon} and Theorem \ref{tm:radonInCap} imply $\eta(B)=0$ and thus $\mu(B)=0$ for every Borel set $B$ with $\caps(B)=0$.
		For showing 
		\begin{equation}\label{eq:iiForM}\mu(B)=\inf\{\mu(A): A \text{ q.o., } B\subseteq A\}\end{equation}
		for every Borel set $B\subseteq \Omega$ with $\mu(B)<\infty$, let the measure $\mu_n$ be defined by $\mu_n(B)\coloneqq \mu(B\cap\{z>\frac{1}{n}\})$. Then it holds
		\[\mu_n(\Omega)=\mu(\{z>\frac{1}{n}\})\le n\eta(\{z>\frac{1}{n}\}) \le n^2\io z \dd\eta =n^2\left((1,z)_{\Ltwo} - (z,z)_W \right)<\infty.\]
		Let $B\subseteq\Omega$ be a Borel set with $\mu(B)<\infty$, so the definition of $\mu$ implies $\caps(B\cap\{z=0\})=0$. Define $B_n\coloneqq B\cap \{\frac{1}{n}<z\leq \frac{1}{n-1}\}$ for $n\ge 2$ and let $B_1\coloneqq \{1<z\}$. Then $\mu(B)=\sum_n\mu(B_n)$. Since $\mu_n(\Omega)<\infty$, for all $\epsilon>0$ and all $n\in\N$ there exists an open set $O_n$ such that $B_n\subseteq O_n\subseteq\Omega$ and $\mu_n(O_n)<\mu_n(B_n)+\frac{\epsilon}{2^n} =\mu(B_n) + \frac{\epsilon}{2^n}$. Define $A_n=O_n\cap\{z>\frac{1}{n}\} $. Since $z$ is q.c.\@, $A_n$ is q.o.\@ and it holds $B_n\subseteq A_n$ and $\mu(A_n) =\mu_n(O_n)<\mu(B_n)+\frac{\epsilon}{2^n}$. Let $A_0=B\cap\{w=0\}$ and $A=\cup_{n\ge 0} A_n$. Then $A$ is q.o., $B\subseteq A$ and $\mu(A)<\mu(B)+\epsilon$. As $\epsilon$ was arbitrary, this shows \eqref{eq:iiForM}. \\
		Next, we show that $z$ solves \eqref{eq:zForCharK}. By definition of $\mu$ it holds
		\[\io z^2 \dd \mu = \int_{\{z>0\}}z^2\dd\mu= \int_{\{z>0\}}z\dd\eta \le \io z \dd\eta = (1,z)_{\Ltwo}-(z,z)_W<\infty,\]
		so $z\in\Lmu$. Furthermore, it holds
		\begin{multline*}
			(z,v)_W+\io zv\dd\mu=(z,v)_W+\int_{\{z>0\}}zv\dd\mu \\= (z,v)_W +\int_{\{z>0\}}v\dd\eta=(z,v)_W+\io v\dd\eta = \io v \dd x.\end{multline*}
		Here, the last equality follows from $v=0$ q.e.\@ on $\{z=0\}$ since $v\in\Lmu$. 
		By Lemma \ref{lm:lambdaIsMu}, we obtain uniqueness of $\mu$.\\
		The last statement of the proposition follows directly from the definition of $\mu$.
	\end{proof}
	
	Next, we prove some auxiliary results.
	
	\begin{lemma}\label{lm:auxConvWzphi}
		Let $\varphi\in \Cinf$, $w_k\rightharpoonup \bar w$ and $z_k \rightharpoonup \bar z$ in $W$ with $w_k,z_k \in\Linf$. Then it holds 
		\[(w_k,z_k\varphi)_W - (z_k, w_k \varphi)_W \to (\bar w,\bar z \varphi)_W - (\bar z, \bar w \varphi)_W. \]
	\end{lemma}
	\begin{proof}
		Let us first consider $W$ equipped with the $H^s(\Omega)$-inner product. We first note that by definition of that inner product, for $z,w\in W\cap \Linf$ it holds 
		\begin{align*} &(w,z\varphi)_W - (z, w \varphi)_W 
			\\&= \io\io\frac{(w(x)-w(y))(z(x)\varphi(x)-z(y)\varphi(y))}{|x-y|^{d+2s}} \dd y\dd x - \io\io\frac{(z(x)-z(y))(w(x)\varphi(x)-w(y)\varphi(y))}{|x-y|^{d+2s}} \dd y\dd x 
			\\ &= \io\io\frac{(w(x)-w(y))\left[(\varphi(x)-\varphi(y))z(y)+ (z(x)-z(y))\varphi(y) + (z(x)-z(y))(\varphi(x)-\varphi(y)) \right]}{|x-y|^{d+2s}} \dd y\dd x \\& -\io\io\frac{(z(x)-z(y))\left[(\varphi(x)-\varphi(y))w(y)+ (w(x)-w(y))\varphi(y) + (w(x)-w(y))(\varphi(x)-\varphi(y)) \right]}{|x-y|^{d+2s}} \dd y\dd x 
			\\ &= \io\io\frac{(w(x)-w(y))(\varphi(x)-\varphi(y))z(y)}{|x-y|^{d+2s}} \dd y\dd x -\io\io\frac{(z(x)-z(y))(\varphi(x)-\varphi(y))w(y)}{|x-y|^{d+2s}} \dd y\dd x.
		\end{align*}
		Here, the last equality holds since all summands in the terms are absolutely integrable.
		Furthermore, it holds by Hölder's inequality and Lemma \ref{lm:phiFracInCinf} that
		\begin{align*}
			&\io \left(\io \frac{(w_k(x)-w_k(y))(\varphi(x)-\varphi(y))}{|x-y|^{d+2s}}\dd y\right)^2 \dd x \\ \le &\io \left( \io \frac{(w_k(x)-w_k(y))^2}{|x-y|^{d+2s}} \dd y\right) \left( \io \frac{(\varphi(x)-\varphi(y))^2}{|x-y|^{d+2s}} \dd y\right) \dd x \\ \le &\left\|\io \frac{\varphi(x)-\varphi(y))^2}{|x-y|^{d+2s}} \dd y\right\|_{\Linf} \io  \io \frac{(w_k(x)-w_k(y))^2}{|x-y|^{d+2s}} \dd y \dd x,
		\end{align*} so 
		\[x\mapsto \io \frac{(w_k(x)-w_k(y))(\varphi(x)-\varphi(y))}{|x-y|^{d+2s}}\dd y \eqqcolon R_{w_k,\varphi}(x)\] is bounded in $L^2(\Omega)$ and therefore w.l.o.g. it holds $R_{w_k,\varphi}\rightharpoonup R$ for some $R\in L^2(\Omega)$. We now want to show $R=R_{\bar w,\varphi}$, with $R_{\bar w,\varphi}$ defined analogously to $R_{w_k,\varphi}$. To do so, we consider the linear and continuous operator 
		\[ \nabla^s: w \mapsto \frac{(w(x)-w(y))}{|x-y|^{d/2+s}}, \qquad \nabla^s \in \mathcal{L}(W, L^2(\Tilde{\Omega}\times \Omega)) \]
		for some $\Tilde{\Omega}\subseteq \Omega$. Weak convergence $w_k\rightharpoonup \bar w$ yields $\nabla^sw_k \rightharpoonup \nabla^s\bar w$ in $L^2(\Tilde{\Omega}\times\Omega)$ and thus
		\[ \int_{\tilde\Omega}\io \frac{(w_k(x)-w_k(y))(\varphi(x)-\varphi(y))}{|x-y|^{d+2s}}\dd y \dd x \to \int_{\tilde\Omega}\io \frac{(\bar w(x)-\bar w(y))(\varphi(x)-\varphi(y))}{|x-y|^{d+2s}}\dd y \dd x  \] 
		By weak convergence of $R_{w_k,\varphi}$ it also holds 
		\[ \int_{\tilde\Omega}\io \frac{(w_k(x)-w_k(y))(\varphi(x)-\varphi(y))}{|x-y|^{d+2s}}\dd y \dd x \to \int_{\tilde\Omega} R\, \dd x.  \]
		Since $\tilde\Omega$ was arbitrary, $R=R_{\bar w,\varphi}$ follows and we have $R_{w_k, \varphi}\rightharpoonup R_{\bar w, \varphi}$ in $\Ltwo$. Analogously, we obtain also $R_{z_k, \varphi}\rightharpoonup R_{\bar z, \varphi}$.
		Putting these results together, using Fubini's theorem and strong convergence of $(w_k), (z_k)$ in $\Ltwo$ yields  
		\begin{align*}
			&(w_k,z_k\varphi)_W - (z_k, w_k \varphi)_W \\ = &\io\io\frac{(w_k(x)-w_k(y))(\varphi(x)-\varphi(y))z_k(y)}{|x-y|^{d+2s}} \dd y\dd x \\ & \,-\io\io\frac{(z_k(x)-z_k(y))(\varphi(x)-\varphi(y))w_k(y)}{|x-y|^{d+2s}} \dd y\dd x 
			\\ \to &\io\io\frac{(\bar w(x)-\bar w(y))(\varphi(x)-\varphi(y))\bar z(y)}{|x-y|^{d+2s}} \dd y\dd x -\io\io\frac{(\bar z(x)-\bar z(y))(\varphi(x)-\varphi(y))\bar w(y)}{|x-y|^{d+2s}} \dd y\dd x 
			\\= &(\bar w,\bar z \varphi)_W - (\bar z, \bar w \varphi)_W.
		\end{align*}
		For the $\tilde H^s(\Omega)$-inner product for $W$, the proof works completely analogous by integrating over $\R^d$ instead of $\Omega$.  In the spectral case $\specHs$ one can proceed similarly as well due to the integral formulation. Note that the part involving $\kappa$ vanishes in $(w_k,z_k\varphi)_W - (z_k, w_k \varphi)_W$. 
	\end{proof} 
	
	Let $\mu_k,\bar\mu \in \M$ and let $z_k, \bar z$ be the solutions of the problems 
	\begin{align}\label{eq:relDirProbZk}
		(z_k,v)_W + \int_{\Omega}z_k v \dd \mu_k &= \int_{\Omega}v \dd x \qquad \forall v\in W\cap L_{\mu_k}^2(\Omega), \\
		\label{eq:relDirProbZ}
		(\bar z,v)_W + \int_{\Omega}\bar z v \dd \bar\mu &= \int_{\Omega}v \dd x \qquad \forall v\in W\cap L_{\bar\mu}^2(\Omega).
	\end{align}
	
	\begin{lemma}\label{lm:wPhiDense}
		Let $\bar\mu\in\M$ and $\bar z$ the solution of problem \eqref{eq:relDirProbZ}. Then the set $\{ \bar z \varphi : \varphi \in \Cinf\}$ is dense in  $W\cap L_{\bar\mu}^2(\Omega)$ .
	\end{lemma} 
	\begin{proof}
		We follow the proof of \cite[Proposition 5.5]{115DirInPerf}.
		Since $\bar z\in\Linf \cap L_{\bar\mu}^2(\Omega)$, we obtain $\bar z \varphi\in W\cap\Linf \cap L_{\mu}^2(\Omega) $ by Lemma \ref{lm:prodOfWFuns} for all $\varphi\in\Cinf$. As every function in $W\cap L_{\bar\mu}^2(\Omega)$ can be approximated by cut-off functions,  
		it suffices to show the existence of an approximating sequence for each $w\in W\cap\Linf \cap L_{\bar\mu}^2(\Omega)$ and $w\ge 0$. 
		For such a $w$, let $w_k$ be the solution of problem \eqref{eq:eqNDiff} in Lemma \ref{lm:eqNDiff}. Then we can employ the comparison principle in Lemma \ref{lm:monOfSolsRelDir} to $w_k$ and $\bar z$ with $\dd\mu_1 = \dd \bar\mu + k\dd x$, $\dd \mu_2 = \dd\bar\mu$, $f_1=kw$ and $f_2=k\|w\|_{\Linf}$ to obtain $w_k\le k\|w\|_{\Linf}\bar z$. Furthermore, Lemma \ref{lm:eqNDiff} yields $w_k\to w$ in $W$ and $L_{\bar\mu}^2(\Omega)$. Hence, we can assume w.l.o.g. that there is $c>0$ such that $0\le w \le c \bar z$.
		It holds $\{(w-c\epsilon) >0 \}\subseteq \{\bar z>\epsilon\}$ 
		and $(w-c\epsilon)_+\to w$ in $W$ and $L_{\bar\mu}^2(\Omega)$ for $\epsilon\to 0$. 
		Thus, we can also assume that $\{w>0\}\subseteq \{\bar z>\epsilon\}$. This assures that $\frac{w}{\bar z}=\frac{w}{\max(\bar z,\epsilon)}$ is in $W\cap\Linf$ by Lemma \ref{lm:zDivW}. By density of $\Cinf$ in $W$, there is a sequence $(\varphi_k)\subset \Cinf$ bounded in $\Linf$ such that $\varphi_k\to \frac{w}{\bar z}$ in $W$ and also q.e.\@ in $\Omega$, so $\bar\mu$-a.e. in $\Omega$.
		As both $\bar z$ and $\varphi_k$ are in $W\cap\Linf$, it holds $\bar z\varphi_k\to \bar z \frac{w}{\bar z}=\bar z$ in $W$. 
		The sequence $\varphi_k$ is bounded in $\Linf$ and converges to $\frac{w}{\bar z}$ $\bar\mu$-a.e., so $z\varphi_k \to \bar z\frac{w}{\bar z}=z$ strongly in $L_{\bar\mu}^2(\Omega)$ by dominated convergence.
	\end{proof}

	\begin{lemma} \label{lm:equivZkMuk}
		Let $(\mu_k)\subset\M$ be a sequence of measures, $\bar\mu \in \M$ and let $z_k \in W\cap L_{\mu_k}^2(\Omega)$ and $\bar z\in W\cap L_{\bar\mu}^2(\Omega)$ be the solutions of the problems \eqref{eq:relDirProbZk} and \eqref{eq:relDirProbZ}. Then these two statements are equivalent:
		\begin{enumerate}[label=(\roman*)]
			\item $z_k\rightharpoonup \bar z$ in $W$
			\item $\mu_k \gammaTo \bar\mu$.
		\end{enumerate}
	\end{lemma}
	\begin{proof}
		We transfer the proof of \cite[Theorem 6.3]{115DirInPerf} (see also \cite[Theorem 4.3]{29capCompactness}) to the fractional setting here.
		To obtain $(i)$ from $(ii)$ one can simply choose $f\equiv 1 $ in the definition of $\gamma$-convergence.\\
		For the reverse direction, let $\varphi \in \Cinf$ and $f\in \Linf$. Let $(w_k)$ be the solutions of \eqref{eq:relDirProbWk}. Then $(w_k)$ is bounded in $W$ as a consequence of estimate \eqref{eq:wkBounded} and therefore w.l.o.g. there is $\bar w\in W$ such that $w_k$ converges weakly in $W$ to some $\bar w$. First note that $|w_k|\leq c z_k$  due to Lemma \ref{lm:monOfSolsRelDir} with $c=\|f\|_{\Linf}$ and therefore also $|\bar w|\leq c \bar z$ q.e.\@ after passing to the limit $k\to\infty$. Testing \eqref{eq:relDirProbWk} and \eqref{eq:relDirProbZk} with $z_k\varphi$ and $w_k\varphi$, respectively, yields
		\[ (w_k,z_k\varphi)_W + \int_{\Omega} w_k z_k\varphi \dd \mu_k = \int_{\Omega} fz_k\varphi  \dd x\]
		and     \[ (z_k,w_k\varphi)_W + \int_{\Omega} w_k z_k\varphi \dd \mu_k = \int_{\Omega} w_k\varphi  \dd x\]
		Subtracting these two equations yields 
		\[ (w_k,z_k\varphi)_W -(z_k,w_k \varphi)_W = \int_{\Omega} (fz_k-w_k)\varphi \dd x, \] 
		which by Lemma \ref{lm:auxConvWzphi} converges to 
		\[ (\bar w,\bar z\varphi)_W -(\bar z,\bar w \varphi)_W = \int_{\Omega} (f \bar z-\bar w)\varphi \dd x. \] 
		Setting $\eta\coloneqq 1-(\bar z,\cdot)_W$, this can be rewritten as 
		\[ (\bar w,\bar z\varphi)_W  + \int_{\Omega} \bar w \varphi \dd\eta = \int_{\Omega} f \bar z\varphi \dd x. \] 
		Since $\bar z\in \K$, $\eta$ is a non-negative Radon measure in $W^*$, so by Theorem \ref{tm:radonInCap} also contained in $\M$. Let  $\bar \mu$ be the measure from Proposition \ref{prop:charK} corresponding to $\bar z$.
		Thus, 
		\begin{gather*} \int_{\Omega} \bar w \varphi \dd \eta = \int_{\Omega \cap \{\bar z>0\}} \bar w\varphi \dd \eta+\int_{\Omega \cap \{\bar z=0\}} \bar w\varphi \dd \eta= \int_{\Omega \cap \{\bar z>0\}} \bar w\varphi \dd \eta  
			\\= \int_{\Omega \cap \{\bar z>0\}} \bar w\varphi \bar z \dd \bar \mu = \int_{\Omega} \bar w\varphi \bar z \dd \bar \mu. \end{gather*}
		Here, the second equality follows from $|\bar w|\leq c \bar z$ and the third one from the characterization of $\bar \mu$ in Proposition \ref{prop:charK}. Since $\varphi$ was arbitrary, it holds 
		\[ (\bar w,\bar z\varphi)_W  + \int_{\Omega} \bar z \bar w \varphi \dd \bar \mu = \int_{\Omega} f \bar z\varphi \dd x \qquad \forall \varphi \in \Cinf. \] 
		By density of $\{\bar z \varphi : \varphi \in \Cinf\}$ in $W\cap L_{\bar \mu}^2(\Omega)$ from Lemma \ref{lm:wPhiDense}, this also holds for all $\varphi \in W\cap L_{\bar \mu}^2(\Omega)$.
		Since the solutions $w_k$ of \eqref{eq:relDirProbWk} depend continuously on $f$ by the estimate $\|w_k\|_W \le \|f\|_{W^*}$ as a result of testing \eqref{eq:relDirProbWk} with $w_k$, it suffices to consider $f$ in the dense subset $\Linf$ of $W^*$ (see Lemma \ref{lm:LinfDenseW}) in the definition of $\gamma$-convergence. 
	\end{proof}
	
	\begin{theorem}\label{tm:MCompact}
		$\M$ is compact with respect to $\gamma$-convergence, i.e. for every sequence $(\mu_k)$ of measures in $\M$ there exists $\bar\mu\in\M$ such that $\mu_k\gammaTo\bar\mu$.
	\end{theorem}
	\begin{proof}
		The proof works just as in \cite[Theorem 4.5]{29capCompactness} and \cite[Theorem 6.5]{115DirInPerf}. Let $(\mu_k)$ be a sequence of measures in $\M$ and let $(z_k)$ be the solutions in $W\cap L_{\mu_k}^2(\Omega)$ of problem \eqref{eq:relDirProbZk}. Proposition \ref{prop:charK} yields $z_k\in\K$, and by compactness of $\K$ there exists a subsequence that converges weakly in $W$ to some $\bar z\in W$. Again by Proposition $\ref{prop:charK}$, there is $\bar\mu\in\M$ such that $\bar z$ solves \eqref{eq:relDirProbZ}. Thus, the result follows from Lemma \ref{lm:equivZkMuk}. 
	\end{proof}
	
	Next, we want to show that $\gamma$-convergence also implies weak convergence of solutions of \eqref{eq:relDirProbWk} with a sequence of strongly convergent right-hand sides.
	\begin{lemma}\label{lm:gammaConvVarRHSHs}
		Let $(\mu_k)$ be a sequence of capacitary measures that $\gamma$-converges to $\Bar{\mu}$,  $(f_k)\subset W^*$ with $f_k \to f$ in $W^*$ and let $(w_k)$ be the solutions of 
		\[ (w_k,v)_W + \int_{\Omega} w_k v \dd \mu_k = \braket{f_k, v}_W \qquad \forall v\in W \cap L_{\mu_k}^2(\Omega).\]  Then it holds $w_k\rightharpoonup \bar w$, where $\bar w$ solves \[ (\bar w,v)_W + \int_{\Omega} \Bar{w} v \dd \bar{\mu} = \braket{f,v}_W \qquad \forall v\in W\cap L_{\bar\mu}^2(\Omega). \]
	\end{lemma} 
	\begin{proof} 
		This proof replaces $H_0^1(\Omega)$ in the proof of \cite[Proposition 4.8]{29capCompactness} by $W$. Let $\tilde w_k$ be the solutions of \eqref{eq:relDirProbZk} with right-hand side $f$. Since $\mu_k \gammaTo \bar\mu$ it holds $\tilde w_k \rightharpoonup \bar{w}$ in $W$ and $\bar W \in L_{\bar \mu}^2(\Omega)$.  Subtracting the equations for $w_k$ and $\tilde w_k$ and using estimate \eqref{eq:wkBounded}, we obtain 
		\[\|w_k-\tilde w_k\|_W \le \|f_k-f\|_{W^*} \to 0,\]
		so $\tilde w_k-w_k\to 0$.
		By weak convergence $\tilde w_k \rightharpoonup \bar w$ it also holds $w_k \rightharpoonup \bar w$ in $W$. 
	\end{proof}

	\section{Capacitary measures in an optimality system}\label{sec:application}
	
	Next, we use capacitary measures to derive a more detailed  optimality condition for a minimization problem. To be more specific, we want to find a decomposition $\bar\lambda = \bar w\bar\mu$ of a multiplier $\bar\lambda$ in the optimality condition of the minimization problem considered in \cite{sparse} for a solution $\bar w$ and a capacitary measure $\bar \mu$.
	This can be done in two ways:  Firstly, by using the compactness result from the previous section, and secondly, by defining the capacitary measure involved in the new optimality condition directly with the already known multiplier. 
	This second approach is similar to the one in \cite[Theorem 14]{137convOc} for the characterization of a subdifferential. 
	
	We consider the problem 
	\begin{equation}\label{eq:sparseOptProb} \underset{w\in W}{\min}\,\, F(w) + \frac{\alpha}{2} \|w\|_W^2+\beta \int_{\Omega} |w|^p \dd x  \end{equation}
	for $p\in[0,1)$ and $\alpha, \beta>0$. For $p\in (0,1)$, this problem was investigated in \cite{sparse}. Note that the solutions of this problem depend on the choice of the norm of $W$. This is why we considered several norms as described in \eqref{eq:equiv_spaces}.\\
	In this section, the following assumptions are supposed to be satisfied.
	
	\begin{assumption}\label{ass:f}
		The function $F: W\to\R$ is weakly lower semicontinuous and bounded from below by an affine function, i.e. there are $g\in W^*$, $c\in \R$ such that $F(w) \geq g(w) + c$ for all $w\in W$. Furthermore, $F\colon W\to\R$ is continuously Fréchet differentiable.
	\end{assumption}
	From now on we also choose $s\in (0,1)$ with $s \neq \frac{1}{2}$, since for $s=\frac12$ in \cite{sparse} the Lions-Magenes space $H_{00}^{\frac{1}{2}}(\Omega)$ was considered instead of the space $H_0^{\frac{1}{2}}(\Omega)$ we use here.
	
	\subsection{Case $p\in (0,1)$}\label{sec:pIn01}
	We first consider the case $p\in (0,1)$.
	Let $\bar w $ be a local solution of problem \eqref{eq:sparseOptProb} and let $\bar \lambda \coloneqq \frac{1}{\beta}[-F'(\bar w)-\alpha(\bar w, \cdot)_W] \in W^*$. Then the following optimality condition is satisfied by \cite[Theorem 5.7]{sparse}:
	\begin{align} \label{eq:ocMinprobSparse}
		\alpha(\bar w,v)_W+\beta\braket{\bar\lambda,v}_W &= -\braket{F'(\bar w),v}_W \qquad \forall v\in W, \\ \label{eq:ocMinProbSparseLambda}
		\braket{\bar\lambda,\bar w}_W &= p \io |\bar w|^p \dd x.
	\end{align}
	This optimality condition is obtained in \cite{sparse} by passing to the limit in the optimality condition of a smoothed auxiliary problem. In that auxiliary problem, the non-smooth and non-convex $L^p$-pseudo norm is approximated by  \begin{align*}
		G_{\epsilon}(w)\coloneqq \int_{\Omega} \psi_{\epsilon}(|w|^2) \dd x,
	\end{align*} 
	where \begin{equation*}
		\psi_{\epsilon}(t) \coloneqq \begin{cases}
			\frac{p}{2}\frac{t}{\epsilon^{2-p}}+(1-\frac{p}{2})\epsilon^p & \textrm{if } t \in [0,\epsilon^2), \\ t^{p/2} & \textrm{if } t \geq \epsilon^2,
		\end{cases}
	\end{equation*}
	with derivative \begin{align*}
		\psi_{\epsilon}'(t)=\frac{p}{2}\min(\epsilon^{p-2},t^{\frac{p-2}{2}}).
	\end{align*}
	For a local solution $\bar w$ of problem \eqref{eq:sparseOptProb}, the auxiliary problem reads
	\begin{equation}\label{eq:auxprob}
		\min_{w\in W}\,\,F(w)+\frac{\alpha}{2}\|w\|_W^2 + \beta G_{\epsilon}(w)+\frac12 \|w-\bar w\|_{\Ltwo}^2 \qquad \text{s.t. } \|w-\bar w\|_W \le \rho
	\end{equation}
	and a local minimum $w_k$ for some $\epsilon_k>0$ satisfies the necessary optimality condition 
	\begin{equation}\label{eq:auxOc}\alpha (w_k,v)_W + \beta \braket{\lambda_k,v}_W +(w_k-\bar w,v)_{\Ltwo}=-\braket{F'(w_k),v}_W \qquad \forall v\in W\end{equation}
	with $\lambda_k\coloneqq G_{\epsilon_k}'(w_k)$ given by \[\braket{\lambda_k,v}_W = \io 2 w_k \psi_{\epsilon_k}'(w_k^2) v \dd x = \io p w_k\min(\epsilon_k^{p-2},|w_k|^{p-2})  v \dd x.\]
	
	By this definition, we can also write $\lambda_k=w_k \mu_k$ for 
	\begin{equation}\label{eq:defMuk}\mu_k \coloneqq p\min(\epsilon_k^{p-2},|w_k|^{p-2}).\end{equation}
	Since $0\le \mu_k \in L^{\infty}(\Omega)$ for all $k \in \N$, it can be considered as an element of $W^*$ with $\braket{\mu_k,v}_W =\int_{\Omega} \mu_k v \dd x$ and also as a non-negative Radon measure by Lemma \ref{lm:etaIsRadon}.
	Therefore, $\mu_k$ is a capacitary measure by Theorem \ref{tm:radonInCap}.  Hence, one can rewrite the optimality condition \eqref{eq:auxOc} of the auxiliary problem as \begin{align}\label{eq:ocAuxWithCap}
		\alpha (w_k,v)_W + \beta \int_{\Omega} w_k  v \dd \mu_k = -\braket{F'(w_k), v}_W - (w_k - \Bar{w},v)_{\Ltwo} \quad \forall v\in W.
	\end{align}
	
	In \cite[Lemma 5.1]{sparse}, $w_k\to \bar w$ in $W$ was proven. \\
	Now we can prove the following theorem.
	\begin{theorem}\label{tm:firstMu}
		Let $\bar w$ be a local solution of \eqref{eq:sparseOptProb} and let $\Bar{\lambda}$ be the multiplier from \eqref{eq:ocMinprobSparse}. Then the multiplier $\bar\lambda$ can be decomposed as $\Bar{\lambda} = \Bar{w} \Bar{\mu}^z$ on $(W\cap L_{\bar\mu^z}^2(\Omega))^*$, i.e.
		\begin{equation}
			\label{eq:ocWithMeasure} \alpha(\bar w,v)_W + \beta \int_{\Omega} \Bar{w} v \dd \bar{\mu}^z =  -\braket{F'(\Bar{w}), v}_W \qquad \forall v\in W\cap L_{\bar\mu^z}^2(\Omega),
		\end{equation}
		where $\bar\mu^z$ is the $\gamma$-limit of the sequence $(\mu_k)$ given by \eqref{eq:defMuk}. \\
		Furthermore,
		\[
		\io \bar w^2 \dd \bar \mu^z = p \io |\bar w|^p \dd x.
		\]\\
		Let $\bar z$ solve \eqref{eq:relDirProbZ} with measure $\bar\mu^z$. Then the measure $\bar\mu^z$ can be written as
		\[ \bar\mu^z(B)= \begin{cases}
			\int_B \frac{\dd\eta}{\bar z} & \text{ if } \caps(B\cap\{\bar z=0\})=0, \\ \infty & \text{ if }\caps(B\cap\{\bar z=0\})>0
		\end{cases}\] with $\eta = \frac{1}{\beta}\left( 1- \alpha(\bar z, \cdot)_W\right)$. 
	\end{theorem}
	By $\Bar{\lambda} = \Bar{w} \Bar{\mu}^z$ on $(W\cap L_{\bar\mu}^2(\Omega))^*$ we mean 
	\[
	\braket{\bar \lambda, v}_W = \io \bar w v \dd \bar \mu \qquad \forall v \in W\cap L_{\bar\mu}^2(\Omega).
	\]
	\begin{proof}
		Let $w_k$ solve \eqref{eq:auxOc} as described above. By \cite[Lemma 5.1]{sparse}, we already know that $w_k \to \Bar{w}$, where $\Bar{w}$ solves  
		\begin{equation} \label{eq:ocWithLambdaInProof}\alpha(\Bar{w},v)_W + \beta \braket{\bar \lambda, v}_W  =  -\braket{F'(\Bar{w}), v}_W \qquad \forall v\in W. \end{equation} 
		For $(\mu_k)$ given by \eqref{eq:defMuk}, Theorem \ref{tm:MCompact} yields the existence of a measure $\bar\mu^z\in\M$ such that w.l.o.g. $\mu_k\gammaTo \bar\mu^z$. 
		By definition of $\gamma$-convergence and Lemma \ref{lm:gammaConvVarRHSHs}, one can pass to the limit in equation \eqref{eq:ocAuxWithCap} to obtain $\bar w\in W\cap L_{\bar\mu^z}^2(\Omega)$ and that $\Bar{w}$ solves \eqref{eq:ocWithMeasure}.
		Comparing \eqref{eq:ocWithLambdaInProof} and \eqref{eq:ocWithMeasure}, we obtain $\Bar{\lambda}=\Bar{w}\Bar{\mu}^z$ in $(W\cap L_{\bar\mu^z}^2(\Omega))^*$. This yields  $	\io \bar w^2 \dd \bar \mu^z = p \io |\bar w|^p \dd x$ using \eqref{eq:ocMinProbSparseLambda}.
		By Proposition \ref{prop:charK} the measure $\bar\mu^z$ can be characterized by $\bar z$ as stated above, incorporating the constants $\alpha, \beta$.
	\end{proof}
	
	Note that $\bar w \in L_{\bar\mu^z}^2(\Omega)$ shows $\{\bar z=0\} \subseteq \{ \bar w = 0\}$. \\
	In Theorem \ref{tm:firstMu}, the measure $\bar\mu^z$ is written in terms of a function $\bar z$. The aim is now to find a representation of $\bar\mu$ with $\bar\lambda=\bar w\bar\mu$ where $\bar\mu$ is characterized by $\bar w$ directly.
	
	In this case, the measure $\bar\mu$ is defined directly via $\bar \lambda$ without using $\gamma$-convergence. However, we require some additional assumptions.
	Note that $\bar\lambda \in C_0(\Omega)^*$ according to \cite[Theorem 5.10]{sparse}. 
	
	\begin{theorem}\label{tm:secondMu}
		Let $\bar w$ be a solution of \eqref{eq:sparseOptProb} solving the system  \eqref{eq:ocMinprobSparse}-\eqref{eq:ocMinProbSparseLambda}. Assume that $\bar\lambda \ge 0$ and $\bar w \ge 0$.
		Then it holds 
		\begin{equation}\label{eq:ocMu2}
			\alpha (\bar w,v)_W +\beta  \io \bar wv\dd\bar\mu^w =-\braket{F'(\bar w),v}_W 
			\qquad \forall v\in W\cap L^2_{\bar\mu^{\bar w}}(\Omega),
		\end{equation}
		where $\bar\mu^w\in\M$ can be expressed with $\bar w$ and $\bar\lambda$ via
		\[ \bar\mu^w(B)= \begin{cases}
			\int_B \frac{\dd\bar\lambda}{\bar w} & \text{ if } \caps(B\cap\{\bar w=0\})=0, \\ \infty & \text{ if }\caps(B\cap\{\bar w=0\})>0.
		\end{cases}\]
		Here, $\bar\lambda\in W^*$ is the measure from the optimality condition \eqref{eq:ocMinprobSparse}, so $\bar\lambda=\frac{1}{\beta}(-F'(\bar w)-\alpha(\bar w, \cdot)_W)$. 
		It holds $\Bar{\lambda} = \Bar{w} \Bar{\mu}^w$ on $(W\cap L_{\bar\mu^w}^2(\Omega))^*$ and
		\[
		\io \bar w^2 \dd \bar \mu^w = p \io |\bar w|^p \dd x.
		\] 
	\end{theorem}
	\begin{proof}
		Let $\bar w$ and $\bar\mu^w$ as defined above. We want to show $\bar\mu^w\in\M$. As $0\le\bar\lambda \in W^*$, for every Borel set $B$ with $\caps(B)=0$ we obtain $\bar\lambda^w(B)=0$ using Lemma \ref{lm:etaIsRadon} and Theorem \ref{tm:radonInCap}. This yields  $\bar\mu^w(B)=0$ for Borel sets with $\caps(B)=0$ by definition of $\bar\mu^w$.
		For showing 
		\[\label{eq:iiForMWithF}\bar\mu^w(B)=\inf\{\bar\mu^w(A): A \text{ q.o., } B\subseteq A\}\]
		for every Borel set $B\subseteq \Omega$ with $\bar\mu^w(B)<\infty$, one can proceed completely analogously to the proof of Proposition \ref{prop:charK} with $\bar w$ instead of $z$.
		Next, we show that $\bar w$ solves \eqref{eq:ocMu2} with measure $\bar\mu^w$. By definition of $\bar\mu^w$ it holds
		\[\io \bar w^2 \dd \bar\mu^w= \int_{\{\bar w>0\}} \bar w^2\dd\bar\mu^w= \int_{\{\bar w>0\}} \bar w\dd\bar\lambda = \io \bar w \dd\bar\lambda = \frac{1}{\beta}( \braket{-F'(\bar w),\bar w}_W-\alpha (\bar w,\bar w)_W)<\infty,\] %use f\ge 0 -> w\ge 0
		so $\bar w\in L_{\bar\mu^w}^2(\Omega)$. Furthermore, it holds for all $v\in L^2_{\bar\mu^w}(\Omega)$
		\begin{align*}\alpha (\bar w,v)_W+\beta \io \bar wv\dd\bar\mu^w &=\alpha (\bar w,v)_W+\beta \int_{\{\bar w>0\}}\bar wv\dd\bar\mu^w =\alpha (\bar w,v)_W + \beta \int_{\{\bar w>0\}}v\dd\bar\lambda \\&=\alpha(\bar w,v)_W+\beta \io v\dd\bar\lambda = \braket{-F'(\bar w),v}_W.
		\end{align*}
		Here, the third equality follows from $v=0$ q.e.\@ on $\{\bar w=0\}$ since $v\in L^2_{\bar\mu^w}(\Omega)$. 
		It thus holds $\bar \lambda = \bar w \bar\mu^w$ on $(W\cap L_{\bar\mu^w}^2(\Omega))^*$ and the last statement follows from \eqref{eq:ocMinProbSparseLambda}.
	\end{proof}
	Hence, there are two ways to obtain a measure $\bar \mu$ with $\bar\lambda = \bar w\bar \mu$: as the $\gamma$-limit of the sequence $(\mu_k)$ arising from the smoothing scheme, or by directly defining $\bar\mu$ via the multiplier $\bar\lambda$. However, the second representation was only shown for $\bar\lambda\ge 0$ and $\bar w\ge 0$. Thus, quite strong assumptions are used in this case. \\
	The next Lemma provides a sufficient condition for the assumptions $\bar\lambda \ge 0$ and $\bar w\ge 0$ in the previous theorem to be satisfied.
	\begin{lemma}
		Let  $F$ satisfy 
		\[w_-\le v_- \Rightarrow F(v) \le F(w)\]
		for $v,w \in W$ with $v=w$ on $\{v\ge 0\} = \{w\ge 0\}$.
		Let $\bar w$ be a local minimum of \eqref{eq:sparseOptProb} and let $w_k$ minimize \eqref{eq:auxprob} for some $\epsilon>0$. Then $\bar w\ge 0$ and $w_k\ge 0$ for $k$ large enough.
		If $w_k\ge 0$ for all $k$ large enough, then $\bar\lambda\ge 0$.
	\end{lemma}
	\begin{proof}	
		By assumption, $\bar w$ is a local solution of \eqref{eq:sparseOptProb}.
		Let now $\delta>0$ and $v\coloneq \bar w_+ + (1-\delta)\bar w_- \in W$. Then  it holds $\|v\|_W^2 \le \|\bar w\|_W^2$ by the same arguments as in Lemma \ref{lm:inequPosPart}. By the assumption on $F$ and $w_- \le v_-$ it holds $F(\bar w) \ge F(v)$. One can also easily observe that $ \io |v|^p \dd x \le  \io |\bar w|^p \dd x$ and  this inequality is strict for $v \neq \bar w$.
		As this holds for all $\delta>0$, local optimality of $\bar w$ yields $\bar w=v$ and thus $\bar w \ge 0$. \\
		To show $w_k\ge 0$, define $v_k \coloneq (w_{k})_+ + (1-\delta)(w_k)_- $. Then for $\bar w \ge 0$ it holds $\| v_k - \bar w\|_{\Ltwo} \le \| w_k - \bar w\|_{\Ltwo}$. For $k$ large enough, it holds $\|w_k-\bar w\|_W < \rho$, and thus also $\| v_k -\bar w\|_W \le \|w_k -\bar w\|_W + \delta \|(w_k)_-\|_W \le \rho$ for $\delta$ small enough.
		Together with the arguments above, this yields $w_k \ge 0$ for each solution $w_k$ of the auxiliary problem \eqref{eq:auxprob}.  
		By definition of $\lambda_k$ it therefore holds $\lambda_k\ge 0$ which yields $\bar\lambda\ge 0$ by passing to the limit $k\to\infty$.
	\end{proof}
	
	\begin{remark}
		The minimization problem considered in \cite{spatSparse} is a time-dependent version of problem \eqref{eq:sparseOptProb} with $\bar\lambda\ge 0$ and $\bar w \ge 0$.
	\end{remark}
	
	The aim is now to investigate when the two measures $\bar\mu^z$ and $\bar\mu^w$ from Theorems \ref{tm:firstMu} and \ref{tm:secondMu} coincide. 
	To do so, we extend some of the previous results to the case of a right-hand side $0\le f \in W\cap\Linf $ instead of $1$. 
	
	\begin{lemma}\label{lm:eqNDiffwithF}
		Let $\mu\in\M$, $w\in W\cap L_{\mu}^2$,  $0\le f \in W\cap\Linf $  and let $w_k\in  W\cap \Lmu$ be the solution of 
		\begin{equation}\label{eq:eqNDiffWithF}
			(w_k,v)_W+\int_{\Omega}w_k v \dd \mu + k \int f(w_k-w)v  \dd x = 0\qquad \forall v\in  W\cap \Lmu.
		\end{equation}
		Then $w_k \to \tilde w$ along a subsequence in $W$ and in $L_{\mu}^2(\Omega)$ with $\tilde w=w$ q.e.\@ on $\{f>0\}$.
	\end{lemma}
	\begin{proof}
		Using the same steps as in the proof of Lemma \ref{lm:eqNDiff} yields
		\[ \frac1{2} \|w_k-w\|_W^2 + \frac{1}{2} \|  w_k-w\|_{L_{\mu}^2(\Omega)}^2 + k \io f(w_k-w)^2 \dd  x \le  \frac1{2} \|w\|_W^2 + \frac{1}{2} \|  w\|_{L_{\mu}^2(\Omega)}^2.\]
		Thus, we obtain $w_k\to w$ in $L^2(\{f>0\})$ and $w_k\rightharpoonup \tilde w$ along a subsequence in $W$ and $L_{\mu}^2(\Omega)$ for some $\tilde w$ with $w=\tilde w$ q.e. on $\{f>0\}$.
		
		Next, we prove $w_k\to \tilde w$ along a subsequence in $W$. Analogous to the procedure to obtain the first equation of this proof, we test \eqref{eq:eqNDiffWithF} with $w_k-\tilde w$ and add some terms to obtain
		\[
		\|w_k-\tilde w\|_W^2 + \int_{\Omega}  (w_k-\tilde w)^2 \dd \mu + k \io f(w_k-w)(w_k-\tilde w)  \dd x = -(\tilde w,w_k-\tilde w)_W - \int_{\Omega}\tilde w(w_k-\tilde w)\dd \mu.
		\]
		Using the already obtained convergence statements, this shows strong convergence $w_k \to \tilde w$ along a subsequence in  $W$ and $\Lmu$.
	\end{proof}
	Here, $\tilde w=w$ q.e.\@ on $\{f>0\}$ is well-defined since we work with the unique q.c.\@ representative of $f\in W$. 
	
	\begin{lemma}
		\label{lm:muInfwithF}
		Let $0\le f\in W\cap\Linf $ %\cap\Linf$
		, $\mu \in\M$ and $w\in W\cap L_{\mu}^2(\Omega)$ be the solution of 
		\[ (w,v)_W + \int_{\Omega} wv \dd \mu = \int_{\Omega} fv \dd x \qquad \forall v \in W\cap L_{\mu}^2(\Omega).\]
		Then $\mu(B)=\infty$ for all Borel sets $B\subseteq \Omega$ with $\caps(B\cap\{w=0\}\cap\{f>0\})>0$.
	\end{lemma}
	\begin{proof} The proof is very similar to the one of Lemma \ref{lm:muInf}. For $z\in W\cap L_{\mu}^2(\Omega)$ with $0\le z\le 1$ q.e.\@ in $\Omega$ and  $z_n\in  W\cap L_{\mu}^2(\Omega)$ being the solution of
		\[(z_n,v)_W + \int_{\Omega} z_n v \dd\mu + n\int_{\Omega} f z_n v \dd x=n\io f zv\dd x \] for all $n\in\N$,
		the comparison principle from Lemma \ref{lm:monOfSolsRelDir}  applied with $f_1=fz, f_2=f, \mu_1=\mu+n\dd x$ and $\mu_2=\mu$ yields $0\le\frac{1}{n}z_n\le w$. Thus, $z_n=0$ q.e.\@ in $\{w=0\}$ and therefore also $z=0$ q.e.\@ in $\{w=0\}\cap \{f>0\}$ by Lemma \ref{lm:eqNDiffwithF}. \\ 
		Following the proof of Lemma \ref{lm:muInf} but working with the set $\{w=0\}\cap \{f>0\}$ shows $\mu(B)=\infty$ on Borel sets $B$ with $\caps(B\cap \{w=0\}\cap \{f>0\})>0$.
	\end{proof}
	
	\begin{lemma}
		\label{lm:lambdaIsMuWithF}
		Let $\mu_1,\mu_2\in\M$, $0\le f\in W\cap\Linf$ 
		and let $w\in W\cap L_{\mu_1}^2(\Omega)\cap L_{\mu_2}^2(\Omega)$ such that
		\begin{align} \label{eq:lambdaIsMu1WithF}
			(w,v)_W+\io wv \dd \mu_1=\io f v \dd x\qquad \forall v\in W\cap L_{\mu_1}^2(\Omega) \\ \label{eq:lambdaIsMu2WithF}
			(w,v)_W+\io wv \dd \mu_2 =\io f v \dd x\qquad \forall v\in W\cap L_{\mu_2}^2(\Omega).
		\end{align} Then $\mu_1(B)=\mu_2(B)$ for all sets $B\subseteq \{w>0\}\cup \{f>0\}$.
	\end{lemma}
	\begin{proof}
		For a Borel set $B$ with $B\subseteq\{w>0\}$, equality of the two measures follows just as in the proof of Lemma \ref{lm:lambdaIsMu}.
		If $B\subseteq \{w=0\}\cap\{f>0\}$ is a Borel set with $\caps(B)>0$, then $\mu_1(B)=\mu_2(B)=\infty$ by Lemma \ref{lm:muInfwithF}. For a general Borel set $B\subseteq \{w>0\} \cup \{f>0\}$, this yields
		\begin{multline*}
			\mu_1(B)=\mu_1(B\cap\{w>0\})+\mu_1(B\cap\{w=0\}\cap \{f>0\})\\=\mu_2(B\cap\{w>0\})+\mu_2(B\cap\{w=0\}\cap \{f>0\})=\mu_2(B).
		\end{multline*}
	\end{proof}
	
	\begin{remark}
		In the previous Lemma \ref{lm:lambdaIsMuWithF} we only proved equality of the two measures on $ \{w>0\} \cup \{f>0\}$. This is due to the fact that in the proof we proceeded as in Section \ref{sec:compact} when working with $f\equiv 1$. There, non-negativity of $f$ was used to obtain comparison principles as for example in Lemma \ref{lm:wNonneg}. When trying to split $f=f_++f_-$ and $w=w_++w_-$ to generalize the result, we were not able to deduce $\mu=\infty$ on $w=0$, as it could be possible that both $w_+$ and $w_-$ are non-zero on $w=0$ and thus $\mu_+,\mu_-<\infty$.
	\end{remark}
	
	\begin{proposition}\label{prop:charKwithF}
		Let $\bar w$ be a solution of \eqref{eq:sparseOptProb} with $0\le -F'(\bar w)\in W\cap\Linf $, $\bar \lambda\ge 0$ and let $\bar w$ solve
		\begin{equation}\label{eq:zForCharKwithF}(\bar w,v)_W + \braket{ \bar \lambda, \bar w}_W = -\io F'(\bar w)v \dd x \qquad \forall v\in W\cap L^2_{\bar\mu}(\Omega).\end{equation}
		Let $\bar\mu^z$ and $\bar\mu^w$ be the measures from Theorems \ref{tm:firstMu} and \ref{tm:secondMu}, respectively.
		Then it holds $\bar\mu^z=\bar\mu^w$ on Borel sets $B \subseteq \{\bar w>0\} \cup \{-F'(\bar w)>0\}$. 
	\end{proposition}
	\begin{proof}
		The equality of the two measures on sets $B \subseteq \{\bar w>0\} \cup \{-F'(\bar w)>0\}$ follows from Lemma \ref{lm:lambdaIsMuWithF}.
	\end{proof}

	\begin{remark}
		As noted in \cite{sparse}, the choice $\bar w=0$ solves the system \eqref{eq:ocMinprobSparse}, \eqref{eq:ocMinProbSparseLambda} with $\bar\lambda=-\frac{1}{\beta}F'(0)$. 
		This holds also true for an optimality condition of the form 
		\[
		\alpha (\bar w,v)_W +\beta  \io \bar wv\dd\bar\mu_{\infty}^w =-\braket{F'(\bar w),v}_W 
		\qquad \forall v\in W\cap L^2_{\bar\mu^{\bar w}}(\Omega),
		\]
		using capacitary measures as above, where we define $\bar\mu_{\infty}$ as 
		\[ \bar\mu_{\infty}(B)= \begin{cases}
			0 & \text{ if } \caps(B)=0 \\ \infty & \text{ if }\caps(B\cap\{\bar w=0\})=\caps(B)>0.
		\end{cases}\]
		In this case, $L_{\bar\mu_{\infty}}^2(\Omega)=\{0\}$.
		Note that this measure is in $\M$ as sets of capacity zero are quasi open, so $\bar\mu_{\infty}$ satisfies $(ii)$ in Definition \ref{def:capMeasure}.
	\end{remark}

	\subsection{Case $p=0$}
	
	We now want to solve
	\begin{align} \label{eq:minprob0}
		\underset{ w\in W}{\min}\, F(w) + \frac{\alpha}{2}\|w\|_W^2+\beta \int_{\Omega}|w|^0\dd x,
	\end{align}
	where we set $0^0=0$. The $L^0$-pseudo norm measures the support of a function, so
	\[\io |w|^0 \dd x = \mathcal{L}({w>0}),
	\]
	where $\mathcal{L}$ denotes the Lebesgue measure.
	This choice of $p=0$ is not covered in \cite{sparse}. However, the case $p=0$ was already investigated in \cite{lp_cont} for the case $s=1$. We therefore first show existence of solutions and optimality conditions for this problem. This is done in two ways: First, by investigating the limit $p\searrow 0$ for solutions of the problem for $p\in(0,1)$ from \cite{sparse}, and then by smoothing the $L^0$-pseudo-norm similarly as in \cite{sparse} for $p\in(0,1)$. The first approach leads to a stronger optimality condition, whereas the second approach is more useful in the numerical implementation. \\
	The $L^0$-pseudo norm is non-convex and not continuous or weakly lower semicontinuous in $L^2(\Omega)$. However, it is lower semicontinuous on $L^2(\Omega)$ and thus weakly lower semicontinuous on $W$ due to the compact embedding of $W$ into $L^2(\Omega)$. 
	
	\begin{theorem}\label{tm:exsol0}
		Under Assumption \ref{ass:f}, the minimization problem \ref{eq:minprob0} admits a solution.
	\end{theorem}
	\begin{proof}
		This follows by standard arguments as in \cite[Theorem 3.1]{sparse} for $p>0$.
		Let $(w_k)$ be a minimizing sequence. Assumption \ref{ass:f} yields boundedness of $(w_k)$ in $W$. Hence, there is $w\in W$ such that  $w_k \rightharpoonup w$ in $W$ and $w_k \to w$ in $L^2(\Omega)$. Hence, \begin{align*}
			\underset{k\to\infty}{\lim \inf} \, F(w_k) + \frac{\alpha}{2}\|w_k\|_W^2+\frac{\beta}{2} \int_{\Omega}|w_k|^0\dd x  \geq F(w) + \frac{\alpha}{2}\|w\|_W^2+\beta \int_{\Omega}|w|^0\dd x
		\end{align*} due to weak lower semicontinuity.
		This shows that $w$ attains the infimum.
	\end{proof}
	
	\subsubsection{Limit $p\to 0$}\label{sec:ocPTo0}
	To obtain an optimality condition for the minimization problem \eqref{eq:minprob0}, we use the results from \cite{sparse} for $p\in(0,1)$. For a sequence of solutions $\bar w_p$ to a slightly adapted version of problem \eqref{eq:sparseOptProb}, we pass to the limit $p\searrow 0$ in the corresponding optimality condition.\\
	We first state the following convergence result.
	
	\begin{lemma}\label{lm:pTo0conv}
		Let $w_k \to w$ in $\Ltwo$ and let $p_k \to 0$. Then \[
		\io |w_k|^{p_k} \dd x \to \io |w|^0 \dd x. \]
	\end{lemma}
	\begin{proof}
		We proceed as in \cite[Theorem 2.8]{138sparseOptLagr}.
		Define \[ N_{\epsilon}(w) \coloneqq \begin{cases}
			1 & \text{if } w>\epsilon, \\
			0 & \text{if } w \le \epsilon.
		\end{cases} \]
		Let $\epsilon>0$ be fixed. Pointwise convergence $N_{\epsilon}(w) |w_k|^{p_k} \to N_{\epsilon}(w) |w|^0$ for $k\to\infty$ yields $\io N_{\epsilon}(w) |w_k|^{p_k} \dd x \to \io N_{\epsilon}(w) |w|^0 \dd x$ by dominated convergence. \\
		In order to pass to the limit $\epsilon \to 0$ in $\io N_{\epsilon}(w) |w|^0 \dd x$, we note that $ N_{\epsilon}(w) |w|^0  \to |w|^0$ pointwise. Thus, the result follows again by dominated convergence.
	\end{proof}
	
	For the sake of completeness, we repeat the following auxiliary lemma from \cite{spatSparse}.
	\begin{lemma}\cite[Lemma 4.2]{spatSparse}\label{lm:auxLimSupInf} 
		Let $N\in \N$.
		Let $(a^1_k), \dots, (a^N_k)$ be  sequences with
		\[
		\liminf_{k \to \infty} a_k^i \in \R \quad \forall i=1,\dots, N
		\]
		and
		\[
		\limsup_{k\to\infty} \left(\sum_{i=1}^N a^i_k \right) \le \sum_{i=1}^N ( \liminf_{k\to\infty} a^i_k ).
		\]
		Then all sequences $(a^1_k), ..., (a^N_k)$ are convergent with limits in $\R$.
	\end{lemma}
	
	Now we can approximate a local solution $\bar w_0$ of \eqref{eq:minprob0} by solutions of problems for $p>0$.
	
	\begin{lemma}\label{lm:convWk0p}
		Let $p_k \to 0$ and let $\bar w_0$ be a local solution of \eqref{eq:minprob0} such that $\bar w_0$ is locally optimal on $B_{\rho}(\bar w_0)$ for some $\rho>0$. Let $\bar w_{p_k}$ denote the solution of 
		\begin{equation}\label{eq:minProbAuxpFor0} \min_{w\in W} F(w) + \frac{\alpha}{2} \|w\|_W^2 +\beta \|w\|_{p_k}^{p_k} + \frac12 \|w-\bar w_0\|_W^2  \qquad \text{s.t. } \qquad \|w-\bar w_0\|_{W} \le \rho. \end{equation}
		Then $\bar w_{p_k}\to \bar w_0$ in $W$.
	\end{lemma} 
	\begin{proof}
		By definition, the sequence $(\bar w_{p_k})$ is bounded in $W$, so  there is $w^*$ such that $\bar w_{p_k}\rightharpoonup w^*$ in $W$ along a subsequence still denoted by $(\bar w_{p_k})$.
		Optimality of $\bar w_{p_k}$ leads to
		\begin{equation}\label{eq:ineqForLimSupConvPLim}
			F(\bar w_{p_k}) + \frac{\alpha}{2} \|\bar w_{p_k}\|_W^2 +\beta \io |\bar w_{p_k}|^{p_k} \dd x + \frac12 \|\bar w_{p_k}-\bar w_0\|_W^2 \le F(\bar w_0)  + \frac{\alpha}{2} \|\bar w_{0}\|_W^2 +\beta \io |\bar w_0|^{p_k} \dd x.
		\end{equation}
		Passing to the limit inferior and using Lemma \ref{lm:pTo0conv} yields 
		\[ F(w^*) + \frac{\alpha}{2} \|w^*\|_W^2 +\beta \io |w^*|^0 \dd x + \frac12 \|w^*-\bar w_0\|_W^2 \le F(\bar w_0)  + \frac{\alpha}{2} \|\bar w_{0}\|_W^2 +\beta \io |\bar w_0|^{0} \dd x.\]
		Local optimality of $\bar w_0$ and weak closedness of the admissible set implies thus $w^*=\bar w_0$. 
		After passing to the limit superior in equation \eqref{eq:ineqForLimSupConvPLim} and some estimations we obtain
		\begin{multline*}
			F(\bar w_0)  + \frac{\alpha}{2} \|\bar w_{0}\|_W^2 +\beta \io |\bar w_0|^0 \dd x = \lim_{k\to\infty} F(\bar w_0)  + \frac{\alpha}{2} \|\bar w_{0}\|_W^2 +\beta \io |\bar w_0|^{p_k} \dd x \\
			\ge \limsup_{k\to\infty}\left(F(\bar w_{p_k}) + \frac{\alpha}{2} \|\bar w_{p_k}\|_W^2 +\beta \io |\bar w_{p_k}|^{p_k} \dd x + \frac12 \|\bar w_{p_k}-\bar w_0\|_W^2 \right) \\
			\ge \liminf_{k\to\infty} F(\bar w_{p_k}) + \liminf_{k\to\infty} \frac{\alpha}{2} \|\bar w_{p_k}\|_W^2 +\liminf_{k\to\infty} \beta \io |\bar w_{p_k}|^{p_k} \dd x + \liminf_{k\to\infty} \frac12 \|\bar w_{p_k}-\bar w_0\|_W^2 \\
			\ge F(\bar w_0)  + \frac{\alpha}{2} \|\bar w_{0}\|_W^2 +\beta \io |\bar w_0|^{p_k} \dd x 
		\end{multline*}
		Hence the assumptions of Lemma \ref{lm:auxLimSupInf} are satisfied which yields $\|\bar w_{p_k}\|_W\to \|\bar w_0\|_W$, so $\bar w_{p_k} \to \bar w_0$ in $W$.
	\end{proof}
	
	Using $\tilde F \colon W \to \R$, $\tilde F(w )\coloneqq F(w) + \frac12 \|w-\bar w_0\|_W^2$ as choice of $F$ in \cite[Theorem 5.7]{sparse} leads to the optimality condition 
	\begin{equation}\label{eq:ocAuxProbpTo0}
		\alpha (\bar w_{p_k},v)_W + \beta \braket{\bar\lambda_{p_k},v}_W + (\bar w_{p_k} - \bar w_0,v)_W = -\braket{F'(\bar w_{p_k}), v}_W \qquad \forall v \in W \end{equation}
	and 
	\begin{equation}\label{eq:pkLambdaW} \braket{\bar \lambda_{p_k}, \bar w_{p_k}}_W = p_k \io |\bar w_{p_k}|^{p_k}\dd x \end{equation}
	for problem \eqref{eq:minProbAuxpFor0} with $\bar \lambda_{p_k} \in W^*$ as $\|\bar w_{p_k}-\bar w_0\|_{W} < \rho$ for $k$ large enough by the previous Lemma \ref{lm:convWk0p}.
	
	\begin{theorem}\label{tm:ocLimPto0}
		Let $\bar w_0$ be a local solution of \eqref{eq:minprob0}. Then $\bar\lambda_0 \in W^*$ defined by
		\begin{equation} \label{eq:ocLimPto0} \alpha (\bar w_0,v)_W + \beta \braket{\bar \lambda_0,v}_W  = -\braket{F'(\bar w_0), v}_W \qquad \forall v \in W\end{equation}
		satisfies 
		\begin{equation} \label{eq:lambdaW0Forp0}
			\braket{\bar \lambda_0, \bar w_0}_W = 0 .
		\end{equation}
	\end{theorem}
	\begin{proof}
		From condition \eqref{eq:ocAuxProbpTo0} and convergence of $\bar w_{p_k}$, one obtains $\bar \lambda_{p_k} \to \bar \lambda_0$ in $W^*$.
		Thus, we can pass to the limit in \eqref{eq:ocAuxProbpTo0} to obtain the optimality condition \eqref{eq:ocLimPto0} for $\bar w_0$. The statement follows now from \eqref{eq:pkLambdaW} by
		\[\braket{\bar \lambda_0,\bar w_0}_W = \lim_{k\to\infty} \braket{\bar \lambda_{p_k}, \bar w_{p_k}}_W = p_k \io |\bar w_{p_k}|^{p_k}\dd x  \to 0,\]
		as $p_k\to 0$ and $\io |\bar w_{p_k}|^{p_k}\dd x$ is convergent by Lemma \ref{lm:pTo0conv} and thus bounded.
	\end{proof}
	
	\begin{remark}
		By \cite[Remark 5.6]{sparse}, it also holds $\braket{\bar \lambda_{p_k},\bar w_{p_k} \varphi}_W = p \io |\bar w_{p_k}|^{p_k} \varphi \dd x$ for all $\varphi \in C_c^{\infty}(\bar\Omega)$ with $\varphi\ge 0$. Hence, using the same arguments as in the previous proof yields also  $\braket{\bar \lambda_0,\bar w_0 \varphi}_W = 0$ for all $\varphi \in C_c^{\infty}(\bar\Omega)$ with $\varphi\ge 0$. 
	\end{remark}

	\subsubsection{Optimality condition with smoothing scheme}\label{sec:p0Smoothing}
	Another way to approximate a solution $\bar w_0$ and to obtain optimality conditions for the minimization problem \eqref{eq:minprob0} is by employing a smoothing scheme for the $L^0$-pseudo norm similarly as in \cite{sparse} for the case $p\in(0,1)$. This approach yields a more direct representation of the approximating sequence $(\lambda_k)$, which will be helpful when computing a decomposition $\bar\lambda_0 = \bar w_0\bar\mu$ numerically later on. We choose a smooth approximation also used in \cite{164gradProjL0}, where we define $\psi_{\epsilon}^0\colon \R \to \R$  as \begin{align*}
		\psi_{\epsilon}^0(w) \coloneqq \frac{w}{w + \epsilon}
	\end{align*}
	with derivative \begin{align}\label{eq:psiDer0}
		{\psi_{\epsilon}^0}'(w)=\frac{\epsilon}{(w + \epsilon)^2}.
	\end{align}
	Then \begin{align*}
		G^0_{\epsilon}(w)\coloneqq \int_{\Omega} \psi_{\epsilon}^0(w^2) \dd x 
	\end{align*} is an approximation of $G_0^0(w) \coloneqq \int_{\Omega}|w|^0 \dd x$. 
	We have the following convergence properties for $G_{\epsilon}^0$.
	\begin{lemma}\label{lm:convSmooth0}
		Let $\epsilon_k \to 0$, $w_k \to w $ in $\Ltwo$. Then
		\[ G_{\epsilon_k}^0(w) \to G_0^0(w) \qquad \text{and} \qquad \liminf_{k\to \infty}\, G_{\epsilon_k}^0(w_k) \ge G_0^0(w).\]
	\end{lemma}
	\begin{proof}
		Convergence $\io \frac{w^2}{w^2 + \epsilon_k} \dd x \to \io |w|^0\dd x $	follows directly from dominated convergence. 
		Let $w_k \to w$ in $\Ltwo$. As $\frac{w_k^2}{w_k^2 + \epsilon_k} \to 1$ if $w>0$ and $\liminf_{k\to \infty} \frac{w_k^2}{w_k^2 + \epsilon} \ge 0 = |w|^0$ if $w=0$, the second statement follows from Fatou's lemma.
	\end{proof}

	For $\epsilon>0$ let us consider the smoothed version of the original objective
	\begin{align} \label{eq:phi_def0}
		\Phi_{\epsilon}^0(w) \coloneq F(w) + \frac{\alpha}{2}\|w\|_W^2+ \beta G^0_{\epsilon}(w).
	\end{align} 
	For a local solution $\Bar{w}_0$ of the non-smoothed problem \eqref{eq:minprob0}, there is $\rho>0$ such that $\Phi_0^0(\Bar{w}_0) \leq \Phi_0^0(w)$ for all $\|w-\Bar{w}_0\|_{W} \leq \rho$. Hence, the  auxiliary problem \begin{align}\label{eq:auxprob0}
		\min_{w\in W} \Phi_{\epsilon}^0(w) + \frac{1}{2} \|w-\Bar{w}_0\|_{L^2(\Omega)}^2   \hspace{5mm}   \textrm{s.t. } \|w-\Bar{w}_0\|_{W} \leq \rho 
	\end{align} has a solution according to the same arguments as in Theorem \ref{tm:exsol0}.
	
	\begin{lemma} \label{lm:conv_w0}
		Let $\bar w_0$ be a local solution of \eqref{eq:minprob0}, let $(\epsilon_k) \subset \R^+$ be a sequence with $\epsilon_k  \to 0$ and let $w_k$ be the solution of problem \eqref{eq:auxprob0} for the smoothing parameter $\epsilon_k$. Then $w_k \to \Bar{w}_0$ in $W$. 
	\end{lemma}
	\begin{proof}
		The proof works similar as in Lemma \ref{lm:convWk0p}. The sequence $(w_k)$ is bounded in $W$ due to the constraint in \eqref{eq:auxprob0}. Hence, there is $w^* \in W$ such that (after extracting a subsequence if necessary) $w_k \rightharpoonup w^*$ in $W$ and $w_k \to w^*$ in $L^2(\Omega)$. By Lemma \ref{lm:convSmooth0}, it holds $\lim_{k\to \infty} G^0_{\epsilon_k}(\Bar{w}_0) = G^0_0(\Bar{w}_0)$ and $\liminf_{k\to \infty} G^0_{\epsilon_k}(w_k) \ge G^0_0(w^*)$, so $\Phi_{\epsilon_k}^0(\Bar{w}_0) \to \Phi_0^0(\Bar{w}_0)$ and weak lower semicontinuity of $F$ yields $\liminf_{k\to \infty} \Phi_{\epsilon_k}^0(w_k)\geq \Phi_0^0(w^*)$. From optimality of $(w_k)$ in \eqref{eq:auxprob0} we get \begin{align}\label{eq:ineqForLimsupConvP0}
			\Phi_{\epsilon_k}^0(\Bar{w}_0)\geq \Phi_{\epsilon_k}^0(w_k) + \frac{1}{2} \|w_k-\Bar{w}_0\|_{L^2(\Omega)}^2.
		\end{align} 
		Passing to the limit on the left and the limit inferior on the right-hand side yields
		\[
		\Phi_{0}^0(\Bar{w}_0)\geq \Phi_0^0(w^*) + \frac{1}{2} \|w^*-\Bar{w}_0\|_{L^2(\Omega)}^2,\]
		so $w^*=\bar w_0$ by optimality of $\bar w_0$ and weak closedness of the admissible set. 
		As in the proof of Theorem \ref{lm:convWk0p}, passing to the limit superior in equation \eqref{eq:ineqForLimsupConvP0} and estimating the right-hand side  also shows that we can apply Lemma  \ref{lm:auxLimSupInf} to obtain strong convergence $ w_k \to \bar w_0$.
	\end{proof}
	
	\begin{lemma}\label{lm:oc}
		Let $w_{\epsilon}$ be the solution of \eqref{eq:auxprob0} for some $\epsilon>0$ with $\|w_{\epsilon}-\Bar{w}_0\|_W < \rho$. Then it holds \begin{align}\label{eq:oc_aux0}
			\braket{F'(w_{\epsilon}),v}_W+\alpha(w_{\epsilon}, v)_W + \beta  {G_{\epsilon}^0}'(w_{\epsilon})z +(w_{\epsilon}-\Bar{w}_0, z)_{L^2(\Omega)} = 0 \hspace{5mm} \forall v\in W.
		\end{align}
	\end{lemma}
	\begin{proof}
		Since the objective function of problem \eqref{eq:auxprob0} is continuously differentiable, a necessary optimality condition is \begin{align*}
			\braket{F'(w_{\epsilon}),(v-w_{\epsilon})}_W +\alpha(w_{\epsilon}, v-w_{\epsilon})_W + \beta {G_{\epsilon}^0}'(w_{\epsilon})(v-w_{\epsilon}) +(w_{\epsilon}-\Bar{w}, v)_{L^2(\Omega)} \geq 0
		\end{align*}   for all  $v\in  W$ with  $\|v-\Bar{w}_0\|_W \leq \rho$. Using the fact that $\|w_{\epsilon}-\Bar{w}_0\|_W < \rho$ leads to the optimality condition \eqref{eq:oc_aux0}.
	\end{proof}
	
	\begin{lemma} \label{lm:mult_bounded0}
		Let $(w_k)$ be the sequence of solutions to the smoothed problem \eqref{eq:auxprob0} corresponding to a sequence $(\epsilon_k) \subset \mathbb{R}^+$ with $\epsilon_k \to 0$ and a local solution $\bar w_0$ and let $\lambda_k \coloneqq {G_{\epsilon_k}^0}'(w_k)$.  Then there is $\bar\lambda_0 \in W^*$ such that $\lambda_k\to \bar\lambda_0$ in $W^*$. Furthermore, 
		\[	\braket{\Bar{\lambda_0},\Bar w_0}_{W} \ge 0   \qquad \text{and} \qquad	\braket{\bar\lambda_0, \varphi \bar w_0}_W \ge 0 \qquad \forall \varphi \in \Cinf, \varphi\ge 0.\] 
	\end{lemma}
	\begin{proof}
		The first part follows from passing to the limit in the optimality condition \eqref{eq:oc_aux0} and $w_k\to\bar w_0$ in $W$ by Theorem \ref{lm:conv_w0}.
		The second statement follows from passing to the limit in
		\[ \braket{\lambda_k, \varphi w_k}_W = \io \frac{2\epsilon_k w_k^2}{(w_k^2+\epsilon_k)^2} \varphi \dd x \ge 0
		\] 
		for $\varphi \equiv 1$ or $\varphi \in \Cinf$ with $\varphi\ge0$. 
	\end{proof}

	Combining the results above, we obtain the following optimality condition for the original problem similar to \cite[Theorem 5.7]{sparse}.
	\begin{theorem}\label{tm:ocSmooth0}
		Let $\Bar{w}_0$ be a local solution of the original problem \eqref{eq:minprob0}. Then $\Bar{\lambda}_0\in W^*$ 
		defined by \begin{align}\label{eq:oc_lim0}
			\braket{F'(\Bar{w}_0),v}_W +\alpha(\Bar{w}_0,v)_W+\beta \braket{\Bar{\lambda}_0,v}_{W} = 0 \qquad \forall v\in W
		\end{align}
		satisfies
		\begin{align} \label{eq:condLambdaWGeq0}	\braket{\Bar{\lambda}_0,\Bar w_0}_{W} \ge 0   	\end{align}
		and 
		\begin{align}\label{eq:oc_lambda0}
			\braket{\bar\lambda_0, \varphi \bar w_0}_W \ge 0 \qquad \forall \varphi \in \Cinf, \varphi \ge 0.
		\end{align}
	\end{theorem}
	\begin{proof}
		Passing to the limit in \eqref{eq:oc_aux0} yields $\lambda_k\to \bar\lambda_0$ with $\bar\lambda_0$ defined by \eqref{eq:oc_lim0}. Thus, \eqref{eq:condLambdaWGeq0} and \eqref{eq:oc_lambda0} follow from Lemma \ref{lm:mult_bounded0}. 
	\end{proof}
	
	When comparing Theorems \ref{tm:ocSmooth0} and \ref{tm:ocLimPto0}, one can observe that inequality \eqref{eq:condLambdaWGeq0} is satisfied with equality in \eqref{eq:lambdaW0Forp0}. Further, note that when considering the same local solution $\bar w_0$ of \eqref{eq:minprob0}, the multipliers from the two theorems coincide. Thus, by passing to the limit for $p\to0$ we were able to obtain a stronger optimality condition. This is because for $p\to 0$, we pass to limit twice in a row: first in the smooth approximation of the $L^p$-pseudo norm and then for $p\to 0$. The direct smoothing of the $L^0$-pseudo norm in this section only uses one limit $\epsilon \to 0$.
	
	\subsubsection{Optimality conditions using capacitary measures}
	Next, we rewrite the optimality conditions obtained in the previous two sections using capacitary measures. To do this, there are multiple possibilities. 
	
	\begin{theorem}\label{tm:capMeasuresFor0}
		Let $\bar w_0$ be a local solution of problem \eqref{eq:minprob0}. Then there is a capacitary measure $\bar \mu$ such that 
		\begin{equation}
			\label{eq:ocWithMeasure0} \alpha(\bar w_0,v)_W + \beta \int_{\Omega} \Bar{w}_0 v \dd \bar{\mu} =  -\braket{F'(\Bar{w}_0), v}_W \qquad \forall v\in W\cap L_{\bar\mu}^2(\Omega) 
		\end{equation}
		with $\bar\lambda_0 = \bar w_0 \bar\mu$ on $(W\cap L_{\bar\mu}^2(\Omega))^*$ for the multiplier $\bar\lambda_0$ from Theorem \ref{tm:ocLimPto0}. Moreover, the measure $\bar\mu$ satisfies 
		\[\io \bar w_0^2 \dd \bar\mu = 0\]
		so that $\bar \mu = 0$ on $\{\bar w_0 \neq 0\}$.\\
		A measure $\bar \mu$ satisfying \eqref{eq:ocWithMeasure0} can be obtained in the following ways:
		\begin{enumerate}[label=(\roman*)]
			\item As the $\gamma$-limit $\bar\mu^z$ of a sequence of measures $\mu_k^z$ from Theorem \ref{tm:firstMu}
			for a sequence $(p_k)\in (0,1)$ with $p_k \to 0$. 
			\item For a sequence $(p_k)\in (0,1)$ with $p_k \to 0$, let $\bar w_{p_k}, \bar \lambda_k$ and $\mu_k^w$ as in Theorem \ref{tm:secondMu}. Let the assumptions of Theorem \ref{tm:secondMu} be satisfied for all $k$, i.e. $\bar w_k, \bar\lambda_k\ge0$. Then  $\bar \mu$ can be obtained as the $\gamma$-limit $\bar\mu^w$ of a subsequence of measures $\mu_k^w$. 
			\item For a sequence of positive numbers $\epsilon_k \to 0$, $\bar \mu$ can be obtained as the $\gamma$-limit $\bar\mu^0$ of a sequence of measures $\mu_k^0 \coloneqq \frac{2\epsilon_k}{(w_k^2+\epsilon_k)^2}$ with $w_k$ solving \eqref{eq:auxprob0}.
			\item If $\bar w_0\ge 0$ and $\bar\lambda_0\ge 0$ with $\bar\lambda_0$ from \ref{tm:ocLimPto0},  define 
			\[
			\bar\mu^{\lambda}(B) = \begin{cases}	
				0 & \text{if } \caps(B \cap \{\bar w_0 = 0\})=0, \\
				\infty & \text{if } \caps(B \cap \{\bar w_0 = 0\})>0.
			\end{cases}
			\]
		\end{enumerate}
	\end{theorem}
	\begin{proof}
		Let $\bar \mu$ such that equation \eqref{eq:ocWithMeasure0} is satisfied. Comparing equations \eqref{eq:ocLimPto0} and \eqref{eq:ocWithMeasure0} yields $\bar \lambda_0 = \bar w_0 \bar\mu$ on $(W\cap L_{\bar\mu}^2(\Omega))^*$. Thus, one obtains  $\io \bar w_0^2 \dd \bar\mu = \braket{\bar \lambda_0, \bar w_0} = 0$ from equation \eqref{eq:lambdaW0Forp0}. This directly leads to $\bar \mu = 0$ on $\{\bar w_0 \neq 0\}$. \\	
		Next we prove existence of such a $\bar\mu$ by constructing it in the four different manners described above. Let  $p_k\to 0$ and let $(\bar w_{p_k})$ be the sequence considered in Section \ref{sec:ocPTo0} using $\tilde f(w) = f(w) + \frac12 \|w-\bar w_0\|_W^2$. 
		Let $(\mu_k^z)$ and $ (\mu_k^w)$ be sequences of capacitary measures corresponding to the measures from Theorem \ref{tm:firstMu} and Theorem \ref{tm:secondMu} for the solutions $\bar w_{p_k}$. Further, let $(\mu_k^0)$ be a sequence of measures defined by $\lambda_k=w_k\mu_k^0 \coloneqq w_k \frac{2\epsilon_k}{(w_k^2+\epsilon_k)^2}$ with $w_k$ solving \eqref{eq:auxprob0}.
		Then one can extract a $\gamma$-convergent subsequence of each of these sequences by Theorem \ref{tm:MCompact}. Statements (i) and (ii) follow then together with the convergence $\bar w_{p_k} \to \bar w_0$ from Theorem \ref{lm:convWk0p}. Statement (iii) follows from $w_k \to \bar w_0$ by Theorem \ref{lm:conv_w0}.
		\\
		For the result (iv) under the additional assumptions $\bar w_0 \ge 0$ and $\bar \lambda_0 \ge 0$ one can proceed as in Theorem \ref{tm:secondMu} using the optimality condition from Theorem \ref{tm:ocLimPto0}. Then one obtains 
		\begin{align*}
			\bar\mu^{\lambda}(B) &= \begin{cases}
				\io \frac{\dd\bar\lambda_0}{\bar w_0} & \text{ if } \caps(B \cap \{\bar w_0 = 0\})=0  \\ 
				\infty &  \text{ if } \caps(B \cap \{\bar w_0 = 0\})>0  
			\end{cases} \\ 
			&=  \begin{cases}	
				0 & \text{if } \caps(B \cap \{\bar w_0 = 0\})=0 \\
				\infty & \text{if } \caps(B \cap \{\bar w_0 = 0\})>0, \end{cases}
		\end{align*}  
		where the second equality follows from $\bar \mu = 0$ on $\{\bar w_0 \neq 0\}$.
	\end{proof}
	
	Note that if $0\le -F'(\bar w_0) \in W\cap \Linf$, the four measures defined in the previous theorem coincide on $\{\bar w_0>0\} \cup \{-F'(\bar w_0)>0\}$ by Lemma \ref{lm:lambdaIsMuWithF}.
	
	\begin{remark}
		One can rewrite equation \eqref{eq:ocWithMeasure0} with measure $\bar\mu^{\lambda}$ from above under the assumptions $\bar w_0\ge 0$ and $\bar\lambda_0\ge 0$ also as 
		\begin{equation*}
			\alpha(\bar w_0,v)_{W}  =  -\braket{F'(\Bar{w}_0), v}_W  \qquad \forall v\in W \textrm{ such that }  v=0 \text{ on } \{\bar w_0 = 0\}, 
		\end{equation*}
		This follows from the fact that $v\in L_{\bar\mu^{\lambda}}^2(\Omega)$ if and only if $v=0$ on $\{\bar w_0 = 0\}$. 
		In \cite[page 10]{35genDerObstacle}, this observation was discussed for the non-fractional case $s=1$.
	\end{remark}

	\section{Numerical Examples}\label{sec:numerics}
	Now, we compute the solutions $\bar w, \bar z$ and also the corresponding multiplier $\bar \lambda$ and the measure $\bar \mu$ for some examples where $F$ is a tracking type functional. The problem is discretized with finite elements and the stiffness matrix for the fractional part is derived as described in \cite{control1dim, fe2d} for the integral fractional Sobolev space $\tilde{H}^s(\Omega)$. The solution is computed using the algorithm developed in \cite{sparse}.
	
	\subsection{One-dimensional Example}\label{sec:num1d}
	We choose the parameters $\alpha=1, \beta=1, p=0.5$, $s=0.1$ and work with the space $\tilde H^s(\Omega)$. Furthermore, $\Omega = (0,1)$ and 
	\[F(w)=\frac12 \|w-w_d\|_{\Ltwo}^2 \qquad \text{with } \qquad w_d(x) = 20 \left(x-\frac12\right)^2.\]
	After the algorithm has stopped at iteration $K$, $\lambda_K$ and $\mu_K$ are computed via their definition with $\psi_{\epsilon_K}'(w_K^2)$, see \eqref{eq:psiDer0}. Using $\mu_K$, we compute $z_K$ as a solution of \eqref{eq:zForCharK}.
	Figure \ref{fig:1dExample} shows the results: On the left, one can observe that the supports of $w_K$ and $z_K$ coincide. Note that in the previous section we have only shown $\operatorname{supp}(\bar w)\subseteq \operatorname{supp}(\bar z)$. The third plot shows that the measures $\mu_k$ tend to infinity where $z_k$ and $w_k$ tend to zero. 
	
	\begin{figure}[h]
		\begin{subfigure}{0.31\textwidth}
			\includegraphics[scale=0.3]{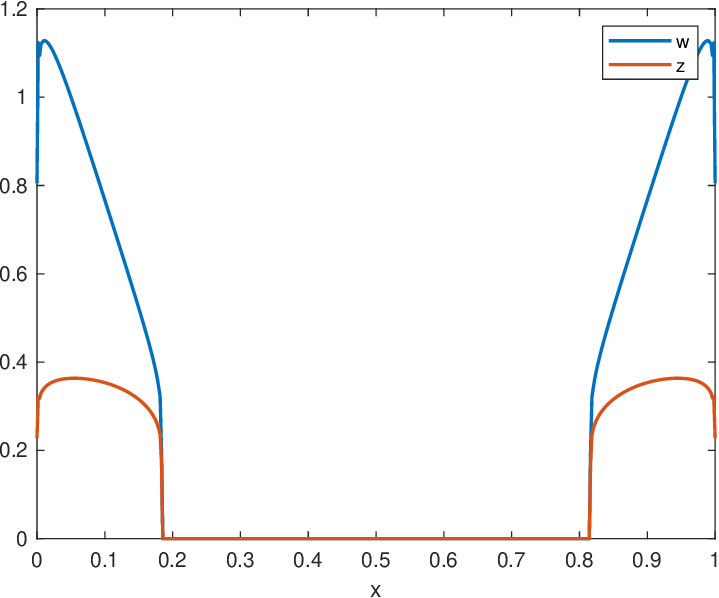} 
			\caption{Solution $w_K$ and $z_K$.}
		\end{subfigure} 
		\begin{subfigure}{0.33\textwidth}
			\includegraphics[scale=0.3]{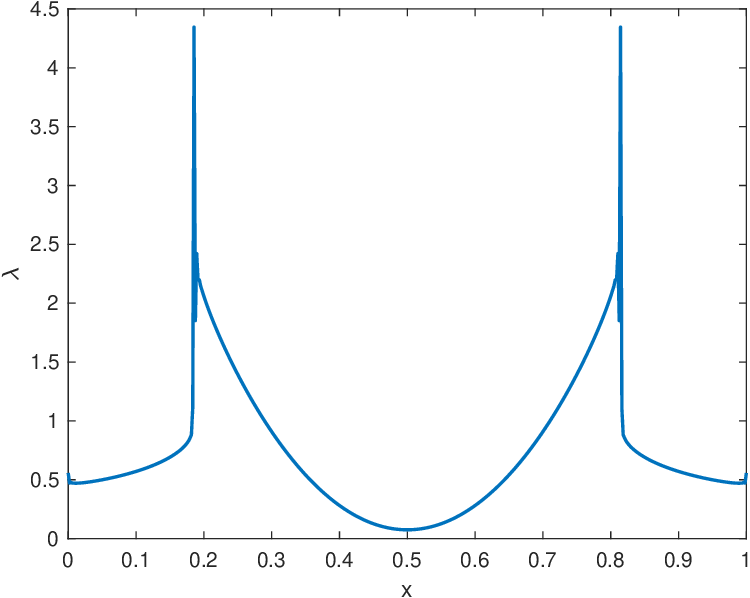}
			\caption{Multiplier $\lambda_K$.}
		\end{subfigure}
		\begin{subfigure}{0.33\textwidth}
			\includegraphics[scale=0.3]{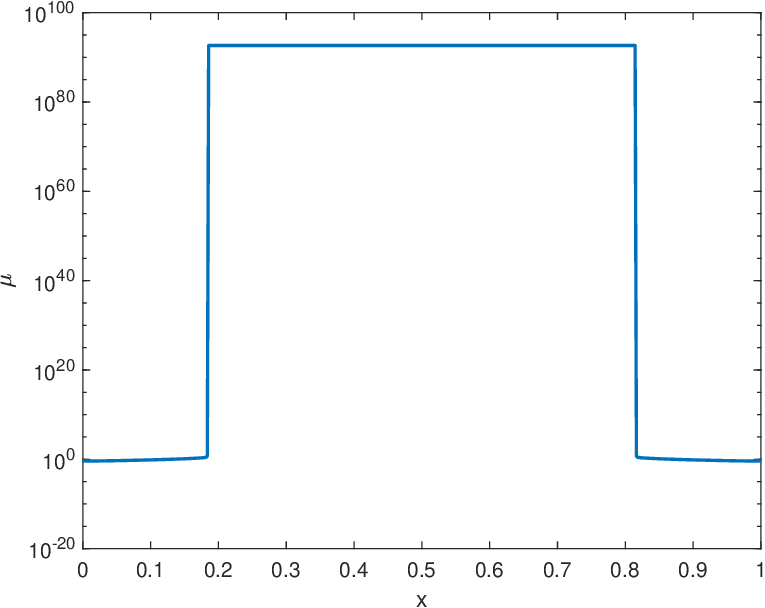}
			\caption{Measure $\mu_K$.}
		\end{subfigure}
		\caption{Solution, multiplier and measure for the one-dimensional example described in Section \ref{sec:num1d}. }\label{fig:1dExample}
	\end{figure}
	
	\subsection{Two-dimensional Example}\label{sec:num2d}
	Next, we consider a two-dimensional example. Here, we use the code from \cite{fe2d} to compute the fractional stiffness matrix for $\tilde H^s(\Omega)$. We choose $\alpha = \beta = 1$, $p=0.05$ and $s=0.1$. Let $\Omega = (-1,1)\times (-1,1)$ and let $F$ be again a tracking type functional with
	\[F(w)=\frac12 \|w-w_d\|_{\Ltwo}^2 \qquad \text{with } \qquad w_d(x) = 5 (|x|^3 + |y|^3).\]
	The results are shown in Figure \ref{fig:2dExample}.

	\begin{figure}[h]
		\begin{subfigure}{0.4\textwidth}
			\includegraphics[scale=0.4]{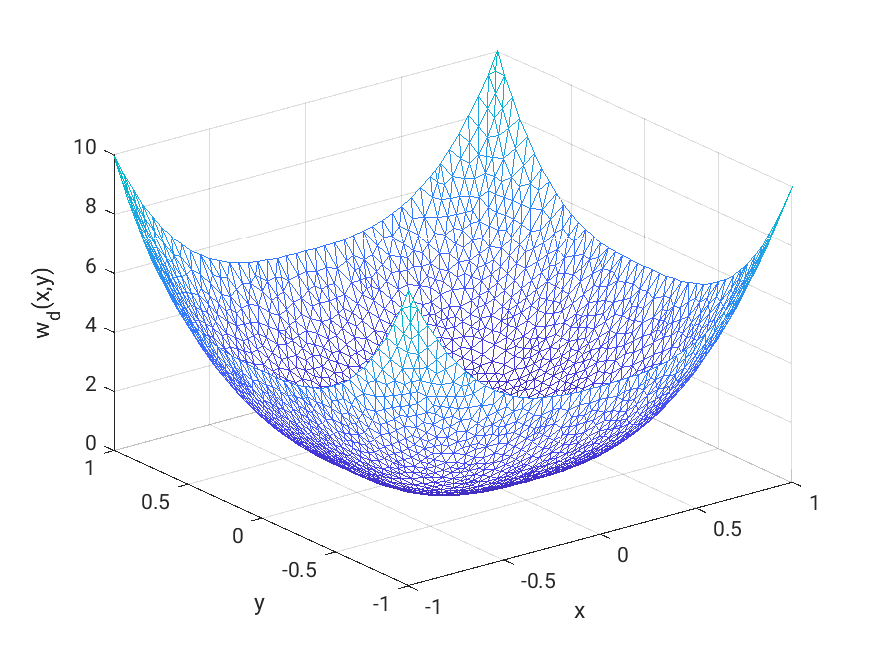} 
			\caption{$w_d$.}
		\end{subfigure}
		\begin{subfigure}{0.4\textwidth}
			\includegraphics[scale=0.4]{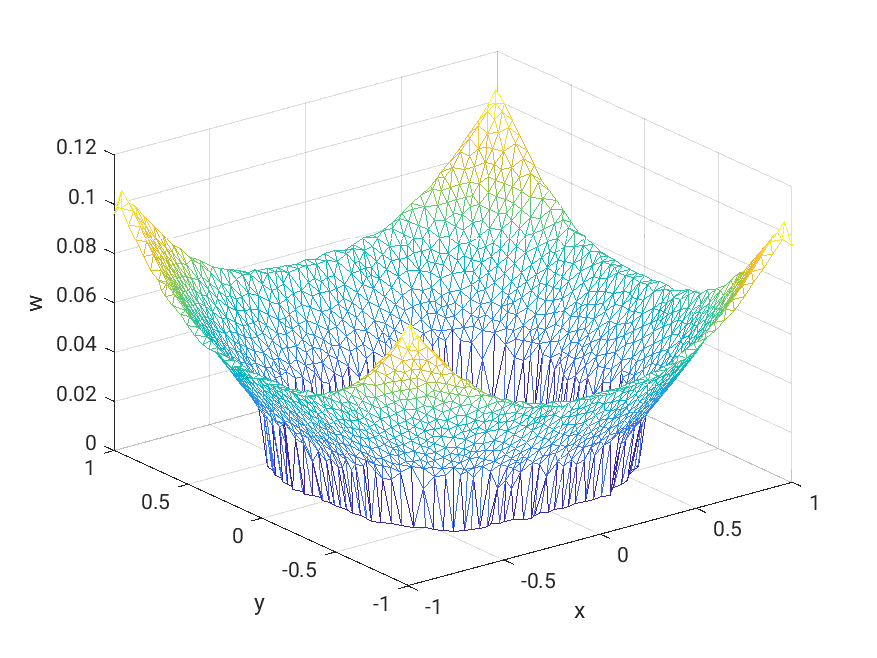} 
			\caption{Solution $w_K$.}
		\end{subfigure} \\
		\begin{subfigure}{0.4\textwidth}
			\includegraphics[scale=0.4]{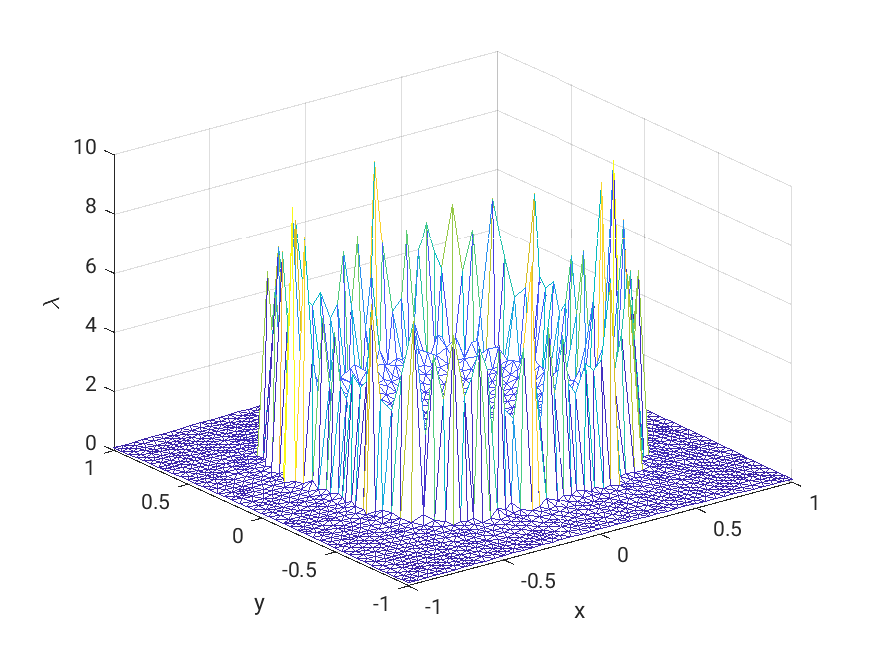}
			\caption{Multiplier $\lambda_K$.}
		\end{subfigure}
		\begin{subfigure}{0.4\textwidth}
			\includegraphics[scale=0.4]{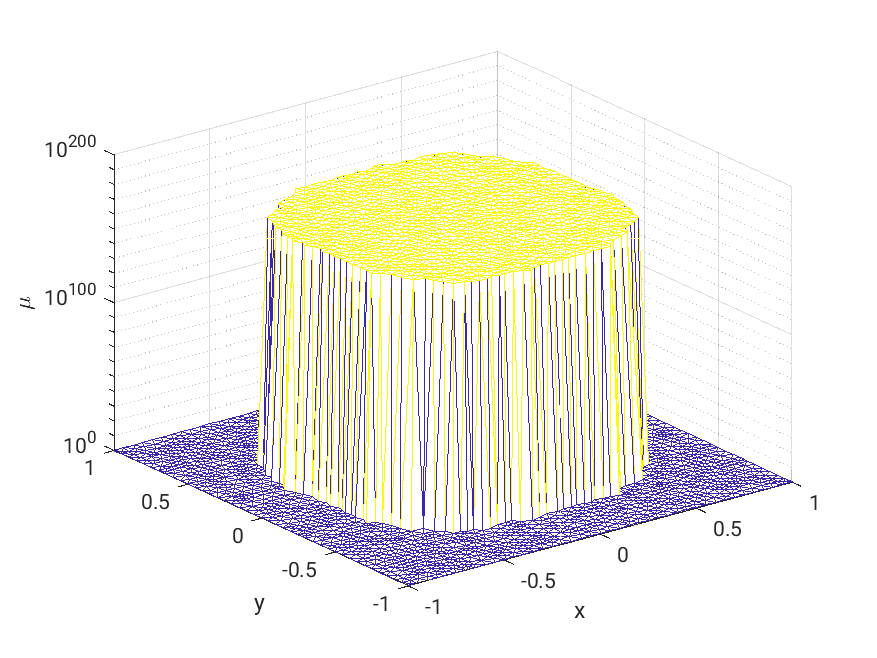}
			\caption{Measure $\mu_K$.}
		\end{subfigure}
		\caption{Solutions, multiplier and measure for the two-dimensional example described in  \ref{sec:num2d}. }\label{fig:2dExample}
	\end{figure}
	
	\subsection{Different fractional Sobolev spaces}\label{sec:compFracDefs}
	To investigate the influence of the choice of the fractional Sobolev space, solutions and their corresponding measures are compared for the spectral and integral spaces $\specHs$ and $\tilde{H}^s(\Omega)$, see Definition \ref{def:specHs} and \eqref{eq:intHs} respectively. For $\Omega = (0,1)$, the parameters $s=0.1$, $p=0.05$ and $\alpha=\beta=1$ and a tracking type functional
	\[ F(w) = \frac12 \|w-w_d\|_{\Ltwo}^2 \qquad \text{with } \qquad w_d(x) = 1.5\,\sin(3\pi x),
	\] the results are shown in Figure \ref{fig:compFracDefs}. The stiffness matrix for the integral case $\tilde{H}^s(\Omega)$ is again computed as in \cite{control1dim}. For $\specHs$, we use the discretization of the spectral fractional Laplacian from \cite{shortmatlab} via its inverse. \\
	One can observe that the solutions differ in their amplitude but have a similar shape. 
	In order to show that the different results are not only a matter of different scaling of the equivalent norms, we also solved the problem for the spectral case $\specHs$ with choice $\alpha = \frac{\|\hat w\|_{\tilde{H}^s(\Omega)}^2}{\|\hat w\|_{\specHs}^2}$, where $\hat w$ is a solution for the $\tilde H^s(\Omega)$-case for $\alpha=1$. However, one can observe that with this choice of $\alpha$ the solutions differ even more.

	\begin{figure}[h]
		\begin{subfigure}{0.45\textwidth}
			\includegraphics[scale=0.4]{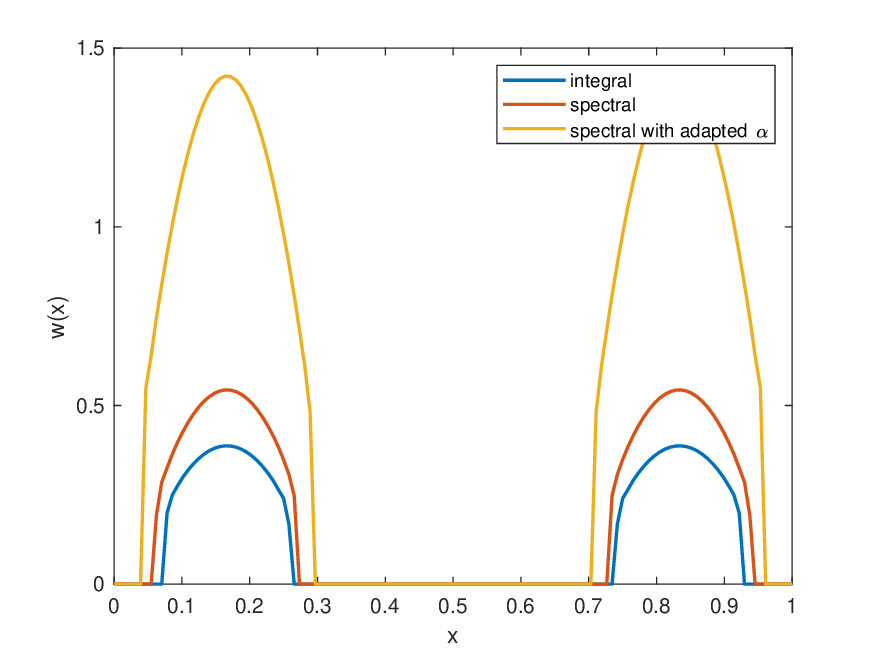} 
			\caption{Solutions $w_K$. 
			}
		\end{subfigure}
		\begin{subfigure}{0.45\textwidth}
			\includegraphics[scale=0.4]{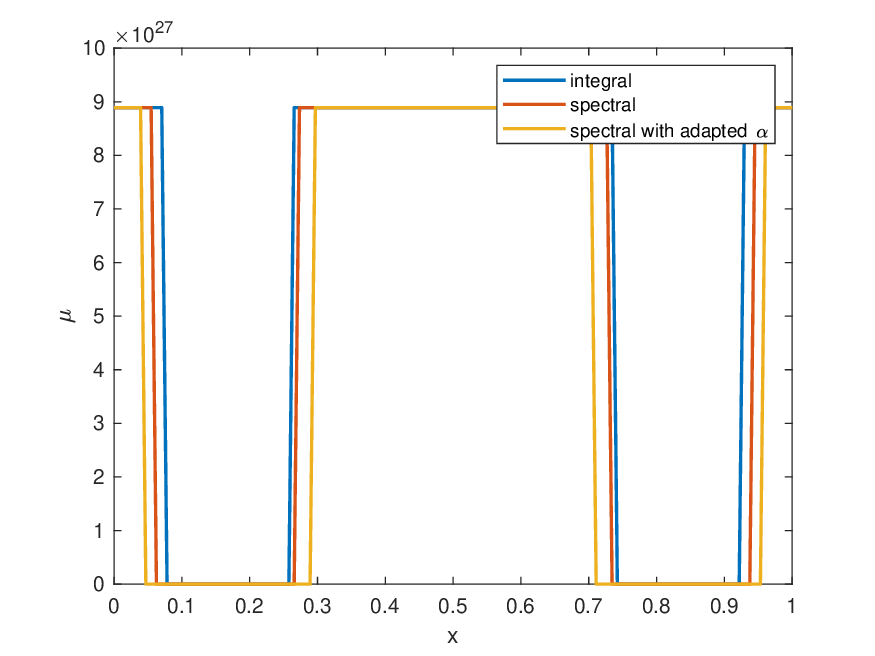} 
			\caption{Measures $\mu_K$.
			}
		\end{subfigure} 
		\caption{Comparison of solutions $w_k$ and measures $\mu_k$ for the integral case $\tilde{H}^s(\Omega)$ and the spectral case $\specHs$ with and without rescaled $\alpha$, see  Section \ref{sec:compFracDefs}. }\label{fig:compFracDefs}
	\end{figure}
	
	\subsection{Case $p=0$}\label{sec:numP0}
	To compute solutions for $p=0$, we use the smoothing scheme from Section \ref{sec:p0Smoothing} again combined with a DC-like algorithm similarly to \cite{sparse}. 
	In each step of this algorithm, we solve the minimization problem
	\[
	w_{k+1} = \underset{w\in W}{\operatorname{argmin}} \, \, F(w) + \frac{\alpha}{2}\|w\|_W^2 + \beta \io \psi_{\epsilon_k}'(w_k^2)w^2.
	\] 
	In this example we choose $s=0.1$, $\alpha=1$, $\beta=0.5$ and $w_d(x) = 10 x (x-1)$ for the tracking functional $F$.\\
	The solution for the choice $\epsilon_k = 0.4^k$ is shown in
	Figure \ref{fig:p0Example}. One can observe that the algorithm converges to a sparser solution for $p=0$ compared to $p=0.1$.\\
	However, the solutions of the algorithm are sensitive to the choice of the sequence $(\epsilon_k)$. For instance, choosing the more slowly decaying sequence $\epsilon_k = 0.9^k$ leads to $w_k \to 0$.\\
	
	\begin{figure}[h]
		\begin{subfigure}{0.45\textwidth}
			\includegraphics[scale=0.4]{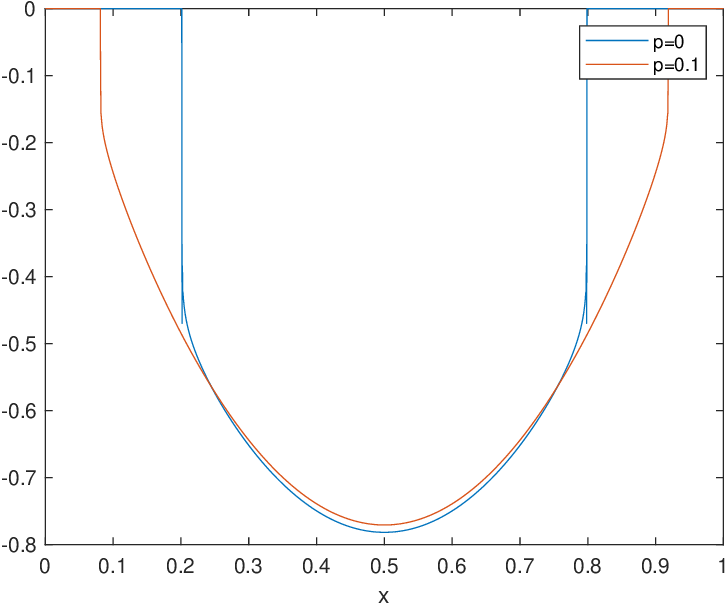} 
			\caption{Solutions $w_K$.}
		\end{subfigure}
		\begin{subfigure}{0.45\textwidth}
			\includegraphics[scale=0.4]{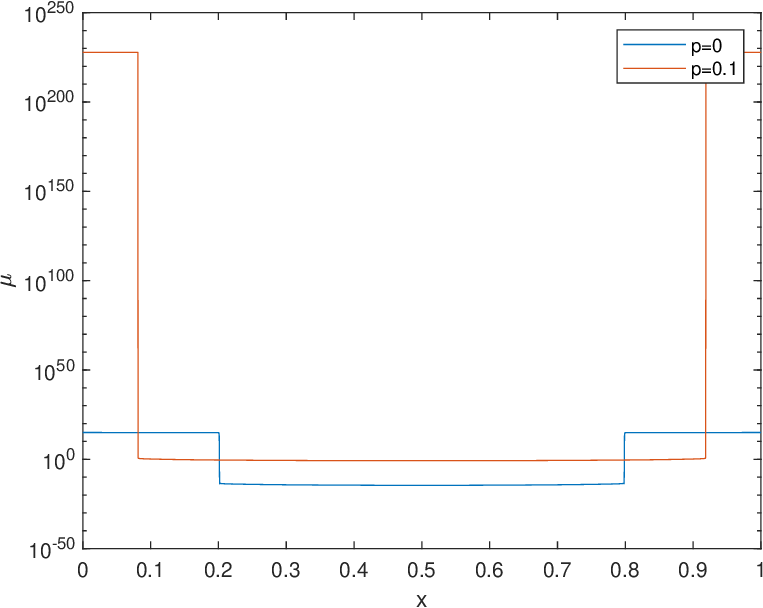} 
			\caption{Measures $\mu_K$.}
		\end{subfigure} 
		\caption{Comparison of solutions and measures for $p=0$ and $p=0.1$ as described in Section \ref{sec:numP0}. }\label{fig:p0Example}
	\end{figure}

	\section*{Acknowledgements}
	The author thanks Daniel Wachsmuth for suggesting the problem analysed in this paper and for many helpful comments.

\end{document}